\documentclass{amsart}  
\usepackage[parfill]{parskip}
\usepackage{graphicx, verbatim}
\usepackage{appendix}
\usepackage{amsfonts}
\usepackage{url}
\usepackage{hyperref} 
\hypersetup{backref,pdfpagemode=FullScreen,colorlinks=true}
\usepackage{amsmath}
\usepackage{amssymb}
\usepackage{tikz-cd}
\usepackage{amscd}
\usepackage{color}
\usepackage{amsthm}
\usepackage{bm}
\usepackage{indentfirst}
\usepackage[hmargin=3cm,vmargin=3cm]{geometry}
\numberwithin{equation}{section}
 
\newtheorem{axiom}[equation]{Axiom} 
\newtheorem{theorem}[equation]{Theorem} 
\newtheorem{proposition}[equation]{Proposition}
 
\newtheorem{lemma}[equation]{Lemma} 
 
\newtheorem{corollary}[equation]{Corollary} 
 
\newtheorem{conjecture}{Conjecture}
\newtheorem{definition}[equation]{Definition}
\theoremstyle{definition}

\newtheorem{notation}{Notation}
\newtheorem{terminology}{Terminology}

\theoremstyle{remark}

\newtheorem{remark}[equation]{Remark}
\newtheorem{example}{Example}
\newtheorem{question}{Question}

\DeclareMathOperator {\spann} {span}

\DeclareMathOperator {\sign} {sign}
\DeclareMathOperator {\Id} {Id}

\DeclareMathOperator {\Symp} {Symp}
\DeclareMathOperator {\Det} {Det}

\DeclareMathOperator {\energy} {energy}

\DeclareMathOperator{\image}{\mathrm{image}}

\DeclareMathOperator{\id}{\mathrm{1}}
\DeclareMathOperator{\lcs}{lcs}

\begin{document}
\title {Elliptic curves in lcs manifolds and metric
invariants}
\author{Yasha Savelyev}
\email{yasha.savelyev@gmail.com}
\address{Faculty of Science, University of Colima, Mexico}

\begin{abstract} 
We study invariants defined by count of charged, elliptic $J$-holomorphic curves in locally conformally symplectic manifolds.
We use this to define $\mathbb{Q} $-valued deformation invariants of certain
complete Riemann-Finlser manifolds and their isometries and this
is used to find some new phenomena in Riemann-Finlser geometry.
In contact geometry this Gromov-Witten theory is used to
study fixed Reeb strings of strict contactomorphisms.
Along the way, we state an analogue of the Weinstein conjecture in lcs
geometry, directly extending the Weinstein conjecture, and
discuss various partial verifications. 
A counterexample for a stronger, also natural form of this
conjecture is given. 


\end{abstract}
\maketitle
\tableofcontents 
\section {Introduction}
The study of $J$-holomorphic curves in symplectic manifolds
was initiated by Gromov ~\cite{cite_GromovPseudo}. In that work and since
then it have been rational $J$-holomorphic curves that were
central in the subject. We study here certain Gromov-Witten
type theory of $J$-holomorphic elliptic curves in locally
conformally symplectic manifolds, for short lcs manifolds.
For lcs manifolds it appears that instead elliptic (and
possibly higher genus) curves
are central. One explanation for this is that rational
$J$-curves in an lcs manifold $M$ have $\widetilde{J}
$-holomorphic lifts to the
universal cover $\widetilde{M}, \widetilde{J}$, where the form is globally
conformally symplectic. Hence, rational
Gromov-Witten theory is a priori insensitive to the information carried by the Lee form (Section
\ref{sec:preliminaries}), although it can still be useful
~\cite{cite_SavelyevLCSnonsqueezing}.   

We will present
various applications for these elliptic curve counts in contact
dynamics and for metric invariants of Riemann-Finlser manifolds. We
choose to start the discussion with applications rather than
theory, as the latter requires certain buildup.
We start with Riemann-Finlser geometry, and this story can be understood as
a generalization of the theory in
~\cite{cite_SavelyevGromovFuller}, which focuses on more
elementary geodesic counts.

In what follows for a manifold $X$ the topology on various functional
spaces is usually the topology of $C ^{0}$ convergence
on compact subsets of $X$, unless specified otherwise. $\pi _{1} (X)$
will denote the set of free homotopy classes of continuous maps $o: S ^{1} \to X$.

We recall some definitions from
~\cite{cite_SavelyevGromovFuller}. 
\begin{definition} \label{definition_boundaryincompressible}
Let $X$ be a smooth manifold. Fix an exhaustion by
nested compact sets $\bigcup _{i \in \mathbb{N}} K _{i}
= X$, $K _{i} \supset K _{i-1}$ for all $i \geq 1$.   We say
that a class $\beta \in \pi _{1} (X)$ is \textbf{\emph{boundary
compressible}}  if $\beta $ is in the image of $$inc _{*}: \pi
_{1}(X - K _{i}) \to \pi _{1} (X)$$  for all $i$, where
$inc: X - K _{i} \to X$ is the inclusion map. 
We say that $\beta $ is \textbf{\emph{boundary incompressible}} if it is not boundary compressible.
\end{definition}
Let $\pi _{1} ^{inc} (X)$ denote the
set of such boundary incompressible classes. When $X$ is
compact, we set $\pi _{1} ^{inc} (X) := \pi _{1} (X)
- const$, where $const$ denotes the set of homotopy classes
of constant loops.

\begin{terminology} All our metrics are Riemann-Finsler
metrics unless specified otherwise, and usually denoted by
just $g$. Completeness, always
means forward completeness, and is always assumed, although
we usually explicitly state this. Curvature always means sectional
curvature in the Riemannian case and flag curvature in the
Finsler case. Thus we will usually just
say complete metric $g$, for a forward complete
Riemann-Finsler metric.  	A reader may certainly choose to
interpret all metrics as Riemannian metrics and completeness
as standard completeness.
\end{terminology}

Denote by $L _{\beta } X$ the class $\beta \in \pi _{1}
^{inc}(X) $ component of the free loop space of $X$, with
its compact open topology. Let $g$ be a complete
metric on $X$, and let $S (g, \beta) \subset L _{\beta } X$ denote the subspace of all
unit speed parametrized, closed $g$-geodesics in class $ \beta$.
The elements of $\mathcal{O} (g, \beta ) =S (g, \beta)
/ S ^{1}$  will be called \textbf{\emph{geodesic strings}}.
A geodesic string will be called \textbf{\emph{non-degenerate}}  if the
corresponding $S ^{1}$ family of geodesics is Morse-Bott
non-degenerate. (Equivalently, the corresponding Reeb orbit in
the unit cotangent bundle is non-degenerate).

\begin{definition} \label{definition_betataut}
We say that a metric $g$ on $X$ is \textbf{\emph{$\beta
$-taut}} if $g$ is complete and $S (g, \beta )$ is compact. We will say that $g$ is \textbf{\emph{taut}} if
it is $\beta $-taut for each $\beta \in \pi _{1} ^{inc}
(X)$. 

\end{definition}
As shown in ~\cite{cite_SavelyevGromovFuller}, a basic
example of a taut metric is a complete metric with non-positive 
curvature, or more generally a complete metric all of
whose boundary incompressible geodesics are minimizing in
their homotopy class. Other substantial classes of examples are 
constructed in ~\cite{cite_SavelyevGromovFuller}.  Overall,
the class of taut metrics is large and flexible. However,
it appears it has not been extensively studied, (or possibly
even explicitly defined). 
\begin{definition}\label{def_Chomotopy} Let 
$ \beta \in \pi_{1} ^{inc}(X) $, and let
$g _{0}, g _{1}$
be a pair of $\beta $-taut metrics on $X$. 
A \textbf{\emph{$\beta $-taut deformation (or homotopy)}} between $g
_{0}, g _{1}$, is a continuous (in the topology of $C ^{0}$
convergence on compact sets) family $\{g _{t}\}$, $t \in
[0,1]$ of complete metrics on $X$, 
s.t. $$S (\{g _{t}\}, \beta ) := \{(o,t) \in L _{\beta }X
\times [0,1] \,|\, o \in S (g _{t}, \beta )\}$$ is compact.
We say that $\{g _{t}\}$ is a \textbf{\emph{taut deformation}}  if it is $\beta $-taut for
each $ \beta  \in \pi _{1} ^{inc}(X)$.  
\end{definition}
As shown in ~\cite{cite_SavelyevGromovFuller}, the  $\beta
$-tautness condition is  
trivially satisfied if $g _{t}$ have the property
that all their class $\beta $ closed geodesics are minimal. In
particular if $g _{t}$ have non-positive curvature then $\{g _{t}\}$ is taut by
the Cartan-Hadamard theorem, ~\cite{cite_ChernBook}.

Let $\mathcal{E} (X)$ be set of equivalence classes of
tuples
$$\{(g, \phi) \,|\, \text{$g$ is a taut and $\phi $ is an
isometry of $g$} \},  
$$
where $(g _{0}, \phi _{0})$ is equivalent to $(g _{1},
\phi _{1})$ whenever there is a Frechet smooth homotopy $\{(g _{t}, \phi
_{t})\} _{t \in
[0,1]}$,  s.t. for each $t$  $\phi
_{t}$ is an isometry of $g _{t}$, and $\{g _{t}\}$ is
a taut homotopy. Let us call such a $\{(g _{t}, \phi
_{t})\} _{t \in [0,1]}$
an \textbf{\emph{$\mathcal{E}$-homotopy}} for future use.

By counting certain charged elliptic curves in a lcs
manifold associated to $(g, \phi) $ we define in Section
\ref{sec:Definition of GWF} a functional:
\begin{theorem} \label{thm:generalizationGWF}
For each manifold $X$, there is a natural,
(generally) non-trivial functional
$$\operatorname {GWF}:  \mathcal{E} (X) \times \pi
_{1} ^{inc} (X) \times \mathbb{N} \to
\mathbb{Q}. $$ 
\end{theorem}
$\operatorname {GWF}$ stands for Gromov-Witten-Fuller, as
when  $\phi=id$ in $\operatorname {GWF}(g, \phi, \beta, n) $ or when $n=0$ the invariant reduces to
a certain geodesic counting invariant studied in
~\cite{cite_SavelyevFuller}, and in this case such counts
can be defined
purely using Fuller's theory ~\cite{cite_FullerIndex}. 

To understand what this functional is counting in general we first define:
\begin{definition}\label{def:geodesicstring} Let $\phi $ be
an isometry of $X,g$.
Then a \textbf{\emph{charge $n$ fixed geodesic string}}  of
$\phi$ is a closed geodesic $o$ whose image is fixed by
$\phi ^{n}$. That is $\image o = \image \phi ^{n} \circ o $.
If the charge is not specified it is assumed to be one. We say
that such a fixed string is in class $\beta $ if the class
of $o$ is $\beta $.
\end{definition}
We will see that if $\operatorname
{GWF} (g, \phi, \beta, n) \neq 0$ then
there is a charge $n$ fixed $g$-geodesic string of $\phi$ in
class $\beta $.
\begin{remark} \label{remark:}
Moreover, $\operatorname
{GWF} (g, \phi, \beta, n)$ is in fact the ``count'' 
of the latter fixed geodesic strings, if by count we mean 
evaluating the fundamental class of a certain compact
virtual dimension zero Kuranishi space with orbifold points.
This will be explained once we construct the
functional as a Gromov-Witten invariant in Section
\ref{sec:Definition of GWF}.
\end{remark}
Here are some basic related phenomena.  Let $ \beta \in \pi
_{1} ^{inc} (X)$ be not a power class (see
~\cite[Definition 1.7]{cite_SavelyevGromovFuller}), and suppose that $X$ admits a $\beta $-taut metric $g$, then by Theorem
~\cite[Theorem 1.10]{cite_SavelyevGromovFuller}
the $S ^{1}$ equivariant homology $H _{*} ^{S ^{1}}(L
_{\beta}X, \mathbb{Z})$ is finite dimensional. In this case we denote by
$\chi ^{S ^{1}}(L _{\beta }X) $ its Euler characteristic.
\begin{theorem} \label{theorem_pertubHyperbolic}
Let $ \beta \in \pi _{1} ^{inc} (X)$ be not a power class, and suppose that $X$ admits a $\beta $-taut
metric. Suppose further that $\chi ^{S ^{1}}(L _{\beta }X)
\neq 0 $. Then for any $\beta $-taut $g$ on $X$, any isometry
$\phi $ of $g$, in the
component of the $id$, has a charge one fixed geodesic
string in class $\beta $. 
\end{theorem}
In the next couple of corollaries,
we need that the manifold $X$ admits a complete metric of
negative curvature. We also need that there is a class
$\beta \in \pi _{1} ^{inc} (X)$. Notably, this condition is
false for $\mathbb{R} ^{n} $. 
\begin{question} \label{question_secondcondition}
Suppose that $X $ is a non simply connected manifold that admits
a complete metric of negative curvature, does one necessarily have $\pi _{1} ^{inc} (X) \neq \emptyset $?
\end{question}
The answer can be shown to be yes in special cases. A simple
case, any hyperbolic, genus at least one,
possibly infinite type Riemann surface satisfies  $\pi _{1}
^{inc} (X) \neq \emptyset $. A three dimensional example:
take $X$ to be the mapping torus by a pseudo-anosov diffeomorphism
of a surface that is the interior of a compact surface with
boundary, with genus at least two. $X$ admits
a hyperbolic metric by Thurston's classification
~\cite{cite_ThurstonClassificationSurfaceDiffeo} and
the geometrization program and it
satisfies $\pi _{1} ^{inc} (X) \neq \emptyset $.

%
\begin{corollary} \label{corollary_hyperbolic} Suppose that $X$
admits a complete metric of negative curvature, and there is a class
$\beta \in \pi _{1} ^{inc} (X)$. Then for any other
complete non-positively curved metric $g$ on $X$, any isometry
$\phi $ of $g$ in the component of the $id$, has a charge
one, class $\beta $ fixed geodesic string. 
\end{corollary}
Theorem \ref{theorem_pertubHyperbolic} deals with isometries
homotopic to the $id$, as such it is interesting only in
dimension three and higher as the topology of surfaces,
admitting metrics with continuous isometry groups is extremely restricted.   The following theorem is about 
general isometries.
\begin{theorem} \label{thm:finitetypeRiemannian}
Suppose that $X$ is a manifold and: 
\begin{itemize}
\item There is an $\mathcal{E}$-homotopy $\{(g _{t}, \phi
	_{t})\}$ on $X$.
\item $g _{0}$ has a unique and non-degenerate geodesic
string in class
$ \beta \in \pi _{1} ^{inc} (X)$. 
\item $\phi _{0,*} ^{n} (\beta) = \beta,  $ for $n \in
\mathbb{N} $.
\end{itemize}
Then there is a class $\beta$, charge $n$, fixed $g
_{1}$-geodesic string of $\phi _{1}$.
\end{theorem}
The following is an immediate corollary, note
that it is non-trivial even for
(infinite type) surfaces.
\begin{corollary} \label{corollary_}
Suppose that $X$ is a manifold and:
\begin{itemize}
\item There is a Frechet smooth homotopy $\{(g _{t}, \phi
	_{t})\} _{t \in [0,1]}$, s.t. $g _{t}$ are complete
	metrics on $X$ and have non-positive curvature.
\item  $\phi _{t} $ is an isometry of $g _{t}$ for each
	$t$.
\item $g _{0}$  has negative curvature.
\end{itemize}
Then for each $\beta \in \pi _{1} ^{inc} (X)$ s.t. $\phi
_{0,*} ^{n} (\beta) = \beta  $ there is a class $\beta$, charge $n$, fixed $g _{1}$-geodesic string of $\phi _{1}$.
\end{corollary}

%

There is a partially related theory of $\phi$-invariant geodesics. The latter are geodesics $\gamma
$ satisfying $\gamma (1) = \phi (\gamma (0) ) $ for some
isometry $\phi$ of $X,g$. (These are also analogous to
translated points of contactomorphisms mentioned ahead.)
A charge 1 fixed geodesic string of $\phi$ clearly determines 
a circle family of closed $\phi$-invariant geodesics. On the other hand as these
$\phi$-invariant geodesics
are not required to be closed, if we fix a $\phi$ they can be shown exist under
very general conditions using Morse theory,
Grove~\cite{cite_GroveGeodesics}. 

\subsection{Contact dynamics, fixed Reeb strings and more applications to
isometries} \label{sec:Fixed Reeb strings}
Let $(C ^{2n+1}, \lambda ) $ be a contact manifold with
$\lambda$ a contact form, that is a one form s.t. $\lambda
\wedge (d \lambda) ^{n} \neq 0$.    Denote by
$R ^{\lambda} $ the Reeb vector field 
satisfying: $$ d\lambda (R ^{\lambda}, \cdot ) = 0, 
\quad \lambda (R ^{\lambda}) = 1.$$ We assume throughout
that its flow is complete.
Recall that a \textbf{\emph{closed $\lambda $-Reeb orbit}}
(or just Reeb orbit when $\lambda $ is implicit)  is
a smooth map $$o: (S ^{1} = \mathbb{R} / \mathbb{Z})    \to
C $$ such 
that $$ \dot o (t) = c R ^ {\lambda}  (o (t)), $$ 
with $\dot o (t) $ denoting the time derivative,
for some $c>0$ called period. Let $S (R ^{\lambda
}, \beta )$ denote the space of all closed Reeb orbits in free
homotopy class $\beta $, with its compact open topology. And set $$\mathcal{O} (R
^{\lambda}, \beta ) := S (R ^{\lambda }, \beta )/S ^{1}, $$
where $S ^{1}$ is acting naturally by reparametrization, see Appendix
\ref{appendix:Fuller}.
We say that the action spectrum is \textbf{\emph{discrete}}
if the image of the period map $A: S (R ^{\lambda }, \beta
) \to \mathbb{R} ^{}  $, $o \mapsto \int _{S ^{0}} o ^{*}
\lambda $ is discrete.

\begin{definition} \label{def:} Let $\phi: (C, \lambda ) \to
(C, \lambda) $ be a strict contactomorphism of a contact
manifold. Then a \textbf{\emph{fixed Reeb string of
$\phi$}} is a closed $\lambda$-Reeb orbit $o$ whose image
is fixed by $\phi$. We say that it is in class $\beta$ if the free homotopy class of $o$ is $\beta $.
\end{definition}

\begin{definition}\label{def:}
Assuming that the class $\beta $ is non-torsion \footnote
{In the torsion case the infinite type condition is more
complicated see ~\cite{cite_SavelyevFuller}.}, we say that
$(C, \lambda) $ is \textbf{\emph{infinite type}} for class
$\beta $ if the action
spectrum of $\lambda$ is discrete and there is
a Reeb perturbation $X$ of the vector field $\mathbb{R}
^{\lambda} $  (in a certain natural
sense, ~\cite[Definition 2.6]{cite_SavelyevFuller}),
s.t. all but finitely many class $\beta$ orbits of $X$ have
even Conley-Zehnder index or or all but finitely many orbits of $X$
have odd Conley-Zehnder index. 
\end{definition} 
A typical example of infinite type is the standard contact form
$\lambda _{st}$ on $S ^{2k+1}$,  as shown in
~\cite{cite_SavelyevFuller}.   
\begin{definition} 
We say that $(C, \lambda ) $ is \textbf{\emph{finite type}}
for class $\beta $ if $\mathcal{O} (R ^{\lambda }, \beta
) $ is compact.  And we say that it is \textbf{\emph{finite
non-zero type}}  if in addition $i (R ^{\lambda}, \beta ) \neq 0$, (the
Fuller index, see Appendix \ref{appendix:Fuller}).
\end{definition}
We have already seen basic examples coming from unit
cotangent bundles of non-positively curved manifolds.  We say that $(C, \lambda
) $ is \textbf{\emph{definite
type}} (for class $\beta$) if it is either finite non-zero type or infinite type.
\begin{theorem} \label{thm:infinitetype}
Let $(C, \lambda ) $ be a contact manifold of
definite type for class $\beta$ orbits, then every strict contactomorphism $\phi $ of $(C, \lambda ) $, homotopic to
the $id$ via strict contactomorphisms, has a fixed
Reeb string in class $\beta $. Furthermore, the same holds for every $\lambda' $ sufficiently $C ^{1}$ nearby to $\lambda $.  
In particular, for any  contact form
$\lambda $ on $S ^{2k+1}$, sufficiently $C ^{1}$ nearby to
$\lambda _{st}$, any strict contactomorphism $\phi $ of $(C,
\lambda ) $ homotopic to the $id$ via strict
contactomorphisms has a fixed Reeb string. 
\end{theorem}

There is a partial connection of the theorem with the theory of
translated points.
\begin{definition} [Sandon~\cite{cite_SandonSheila}]
\label{def:translated}
Given a (not necessarily strict) contactomorphism $\phi$ of
$(C, \lambda ) $, a point $p \in C$ is called
a \textbf{\emph{translated point}}  provided that $\phi ^{*}
\lambda (p) = \lambda (p)  $ and $\phi (p)$ lies
on the $\lambda$-Reeb flow line passing through $p$.
\end{definition}
A fixed Reeb string for $\phi$ in particular
determines a special translated point of $\phi$  (one for
each point on the image of the fixed Reeb
string). So the above theorem is partly related to
the Sandon conjecture ~\cite{cite_SandonSheila} on
existence of translated points of contactomorphisms.
However, also note  that the general form of Sandon's
conjecture has counterexamples on $S ^{2k+1}$ for the
standard contact form $\lambda _{st}$, see
	Cant~\cite{cite_cant2022contactomorphisms}. Partially related to
the Sandon conjecture is the Conjecture
\ref{conj:conformalWeinstein} in Section \ref{sec:Results on Reeb 2-curves}, which is an analogue in lcs geometry of the Weinstein conjecture.

\begin{corollary} \label{cor:Riemannian} Let $X,g$ be  
complete, with a class $\beta \in \pi
_{1} ^{inc} (X)$, and such that its unit cotangent bundle is
definite type for class $\widetilde{\beta} $, (defined as in
Section \ref{sec:Definition of GWF}). Then every isometry of $X,g$ homotopic through
isometries to the $id$ has a class $\beta $ fixed
geodesic string. 
\end{corollary}
\begin{theorem} \label{thm:basis-1}
Suppose that $(C, \lambda) $ is
Morse-Bott and some connected component $N \subset \mathcal{O} (R
^{\lambda}, \beta) $ has non-vanishing Euler characteristic. Then  any  contact form
$\lambda'$ on $C$, sufficiently $C ^{1}$ nearby to $\lambda$,
any strict contactomorphism $\phi $ of $(C, \lambda') $,
homotopic to the $id$ via strict contactomorphisms has
class $\beta $ fixed Reeb string.
\end{theorem}
Both of the theorems above are actually special cases of the
next theorem proved in Section \ref{sec:Proofs of theorems on Conformal symplectic Weinstein conjecture}. For more details on the  Fuller index see Appendix
\ref{appendix:Fuller}. Let $\lambda $ be a contact form on
a closed manifold $C$, $N \subset \mathcal{O} (R ^{\lambda }, \beta )$ and let $i (N, R ^{\lambda}, \beta) \in
\mathbb{Q}  $ denote the Fuller
index.  For example, if $\lambda $ is Morse-Bott (see
~\cite{cite_FredericBourgeois})  and $N$ is a connected component of
$\mathcal{O} (R ^{\lambda}, \beta) $ then by a computation
in ~\cite [Section 2.1.1] {cite_SavelyevFuller} $i (R
^{\lambda }, N, \beta) \neq 0$  if $\chi (N) \neq 0$ (the
Euler characteristic). 


\begin{theorem}  \label{thm:basic0}
Let $(C, \lambda)$ be a contact manifold satisfying
the condition: $i (N, R ^{\lambda}, \beta) \neq 0$, for some open compact $N \subset \mathcal{O} (R ^{\lambda }, \beta) $. Then any strict contactomorphism $\phi: (C, \lambda) \to (C, \lambda)$, homotopic to the $id$ via strict contactomorphisms has
a fixed Reeb string $o$ in class $\beta$ and moreover $o \in N$. 
\end{theorem}
We have already mentioned that the index assumption of the
theorem holds for Morse-Bott contact forms $\lambda $, provided the Euler characteristic of some component of $N \subset
\mathcal{O} (R ^{\lambda }) $ is non-vanishing.   We may
take for instance the standard contact form $\lambda _{st}$
on $S ^{2k+1}$, the unit contangent bundle of the sphere, or see 
Bourgeois~\cite{cite_FredericBourgeois} for more examples.
In this Morse-Bott case the theorem may be verified by elementary considerations. 
To see this suppose we have a connected component $N \subset
\mathcal{O}  (R ^{\lambda}) $  with $\chi (N) \neq 0$.
Then $\phi$ as above induces a topological
endomorphism $\widetilde{\phi} $ of $N$ with non-zero
Lefschetz number, so that in this case the result follows by the Lefschetz fixed point theorem.  

In general a compact open component $N \subset \mathcal{O}
(R ^{\lambda}) $ may not be a finite simplicial complex, or
indeed any kind of topological space to which the classical
Lefschetz fixed point theorem may apply.  Also the relationship of $i (N, R ^{\lambda}, \beta) $ with $\chi (N) $ breaks down in
general as $i (N, R ^{\lambda}, \beta) $  is partly
sensitive to the dynamics of $R ^{\lambda }$.

The following is a variation of Theorem
\ref{theorem_pertubHyperbolic} in the absence of the
condition that $\beta $ be not a power, and removing all
assumptions on the metric except completeness.  This is proved in
Section \ref{sec:Proofs of theorems on Conformal symplectic Weinstein conjecture}.
\begin{theorem} \label{thm:basic1}
Let $X$ admit a complete metric with
a unique and non-degenerate geodesic string in class $ \beta \in
\pi _{1} ^{inc} (X)$.
Then one of the following alternatives holds:
\begin{enumerate}
\item Sky catastrophes for families of Reeb vector fields
exist, and the sky
catastrophe can be essential, see Definition \ref{def:bluesky}.
\label{alt:0}
\item  For any complete metric $g$ on $X$ and every isometry $\phi$
of $X,g$ homotopic through isometries to the identity,
$\phi$ has a charge 1 fixed geodesic string in class $\beta $.
\end{enumerate}
\end{theorem}
\subsection{Conformal symplectic Weinstein conjecture} \label{sec_Conformal symplectic Weinstein conjecture}
We introduce in Section \ref{sec:preliminaries} certain
analogues of Reeb orbits for lcs manifolds. In particular,
we define a unifying concept of a Reeb 2-curve on which most
of the subsequent theory is based.
This leads us to state one analogue
in lcs geometry of the classical Weinstein conjecture, and
we discuss certain partial verifications.
We also state in this section an important
counterexample for a stronger, but also natural form of the lcs Weinstein conjecture.
\subsection{Organization} \label{sec_Organization}
The main theorems are proved in Section \ref{sec:Proofs of theorems on Conformal symplectic Weinstein conjecture}. 
Section \ref{sec:preliminaries} presents detailed preliminaries 
for lcs geometry, which should make this paper self
contained and accessible to a general reader.   Section
\ref{sec:gromov_witten_theory_of_the_lcs_c_times_s_1_}
defines the Gromov-Witten invariant $\operatorname {GWF} $,
which is central to the applications in Riemann-Finlser
geometry.

\section{Background and preliminaries} \label{sec:preliminaries}
\begin{definition}
A locally conformally symplectic manifold or just 
an $\lcs$ manifold, is a smooth $2n$-fold $M$ with an 
$\lcs$ structure: which is a non-degenerate 2-form 
$\omega$, with the property that for every $p \in M$ there is an open $U \ni p$ such that $\omega| _{U} = f _{U} \cdot \omega _{U} $, for some symplectic form $\omega _{U} $ defined on $U$ and some smooth positive function $f _{U} $ on $U$.
\end{definition}  
These kinds of structures were originally considered by Lee
in \cite{cite_Lee}, arising naturally as part of an abstract study of ``a kind of even dimensional Riemannian geometry'', and then further studied by
a number of authors see for instance, \cite{cite_BanyagaConformal} and
\cite{cite_VaismanConformal}.
An $\lcs$ manifold admits all the interesting classical
notions of a symplectic manifold, like Lagrangian
submanifolds and Hamiltonian dynamics, while at the same
time forming a much more flexible class. For example
Eliashberg and Murphy show that if a closed almost complex
$2n$-fold $M$ has $H ^{1} (M, \mathbb{R}) \neq 0 $ then it
admits a $\lcs$ structure,
\cite{cite_EliashbergMurphyMakingcobordisms}. 
Another result of Apostolov, Dloussky
\cite{cite_ApostolovStructures} is that any complex surface with an odd first Betti number admits a $\lcs$ structure, which tames the complex structure.  


To see the connection with the first cohomology 
group $H ^{1} (M, \mathbb{R} ^{} ) $, mentioned 
above, let us point out right away the most basic 
invariant of a $\lcs$ structure $\omega$, when $M$ 
has dimension at least 4. This is the Lee class, 
$\alpha = \alpha _{\omega}  \in H ^{1} (M, 
\mathbb{R}) $. This class has the property that on 
the associated $\alpha$-covering space (see proof of Lemma \ref{lemma:Reeb}) $\widetilde{M} $, the lift $\widetilde{\omega} $ 
is globally conformally symplectic. 
Thus, an $\lcs$ form is globally conformally 
symplectic, that is diffeomorphic to $e ^{f} \cdot 
\omega'$, with $\omega'$ symplectic, iff its Lee class vanishes.

Again assuming $M$ has dimension at least 4, the Lee class 
$\alpha$ has a natural differential form representative, 
called the Lee form, which is defined as follows.  We take
a cover of $M$ by open sets $U _{a} $ in which $\omega=
e ^{f _{a}} \cdot \omega _{a}   $ for $\omega _{a}  $ symplectic.
Then we have 1-forms $d (f _{a} )$ on each $U 
_{a} $, which glue to a well-defined closed 1-form 
on $M$, as shown by Lee.  We may denote this 1-form and its cohomology class both by $\alpha$. It is moreover immediate that for an $\lcs$ form $\omega$, $$d\omega= \alpha \wedge \omega,$$ for $\alpha$ the Lee form as defined above.

As we mentioned $\lcs$ manifolds can also be understood to generalize contact manifolds. This works as follows.
First we have a class of explicit examples of $\lcs$
manifolds, obtained by starting with a symplectic cobordism
(see \cite{cite_EliashbergMurphyMakingcobordisms}) of a closed contact manifold
$C$ to itself, arranging for the contact forms at the two
ends of the cobordism to be proportional and then gluing the
boundary components, (after a global conformal rescaling of the
form on the cobordism, to match the boundary conditions). 
\begin{terminology}
   \label{notation:contact} For us a contact 
   manifold is a pair $(C,\lambda) $ where $C$ is 
   a closed manifold and $\lambda$  a contact 
   form: $\forall p \in C: \lambda \wedge \lambda ^{2n} (p) \neq 0 $. This is not a completely common 
   terminology as usually it is the equivalence 
   class of $(C,\lambda) $ that is called a 
   contact manifold, where $(C,\lambda) \sim (C, 
   \lambda') $  if $\lambda=f \lambda'$ for $f$  
   a positive function. (Given that the contact structure, in 
   the classical sense, is co-oriented.)  A 
   \textbf{\emph{contactomorphism}}  between $(C 
   _{1}, \lambda _{1}) $, $(C _{2}, \lambda _{2}) 
   $  is a diffeomorphism $\phi: C _{1} 
   \to C _{2} $ s.t. $\phi ^{*} \lambda _{2} = f 
   \lambda _{1}$ for some $f>0$. It is called 
   \textbf{\emph{strict}}  if $\phi ^{*} \lambda  _{2} = \lambda _{1}$. 
\end{terminology}
A concrete basic example, which can be understood as a special
case of the above cobordism construction, is the following. 
\begin{example} [Banyaga] \label{example:banyaga} 
Let $(C, \lambda) $ be a contact manifold, $S ^{1}
= \mathbb{R} / \mathbb{Z}   $, $d \theta $ the
standard non-degenerate 1-form on $S ^{1}$ satisfying $\int _{S ^{1}} d \theta = 1$. 
And take $M=C \times S ^{1}  $ with the 2-form $$\omega _{\lambda} = d_{\alpha} \lambda : = d \lambda - \alpha \wedge \lambda,$$ for $\alpha: = pr _{S ^{1} } ^{*} d\theta   $, $pr _{S ^{1} }: C \times S ^{1} \to S ^{1}  $ the projection,  and $\lambda$ likewise the pull-back of $\lambda$ by the projection $C \times S ^{1} \to C $. We call $(M,\omega _{\lambda} )$  as above the \textbf{\emph{lcs-fication}} of $(C,\lambda)$.
This is also a basic example of a first kind lcs 
manifold, as in Definition \ref{def:firstkind} ahead.
\end{example}

   %

The operator 
\begin{equation}
   \label{eq:Lichnerowicz}
  d_{\alpha}: \Omega ^{k} (M) \to \Omega ^{k+1} (M) 
\end{equation}
is called the Lichnerowicz differential with respect to a closed 1-form $\alpha$,
and it satisfies $d_{\alpha} \circ d_{\alpha} =0  $ so that we have an associated Lichnerowicz chain complex.

\begin{definition}\label{def:exactlcs}
An \textbf{\emph{exact lcs form}} on $M$  is an lcs 2-form 
s.t. there exists a pair of one forms $(\lambda, \alpha)$
with $\alpha$ a closed 1-form, s.t. $\omega=d_{\alpha}
\lambda $ is non-degenerate.  In the case above we also call
the pair $(\lambda, \alpha)$ \textbf{\emph{an exact lcs structure}}. The
triple $(M, \lambda, \alpha ) $ will be called an \textbf{\emph{exact lcs
manifold}}, but we may also call $(M, \omega ) $ an exact lcs
manifold when $(\lambda, \alpha ) $ are implicit.
\end{definition}
An exact lcs structure determines a generalized distribution
$\mathcal{V} _{\lambda} $ on $M$: $$\mathcal{V} _{\lambda}
(p) = \{v \in T _{p} {M} \,|\,  d \lambda (v, \cdot) = 0 \},     
$$ which we call the \textbf{\emph{vanishing distribution}}. 
We also define a generalized distribution $\xi _{\lambda}
$ that is the $\omega$-orthogonal complement to $\mathcal{V} _{\lambda}$, which we call \textbf{\emph{co-vanishing distribution}}. For each $p \in M$, $\mathcal{V} _ {\lambda} (p) $  has dimension at most 2 since $d\lambda - \alpha \wedge \lambda$ is non-degenerate. If $M ^{2n} $ is closed $\mathcal{V} _{\lambda} $ cannot identically vanish since $(d\lambda) ^{n} $ cannot be non-degenerate by Stokes theorem.

\begin{definition} \label{def:firstkind} Let $(\lambda,
\alpha) $ be an exact lcs structure on $M$. We call 
$\alpha$ integral, rational or irrational if its periods are
integral, respectively rational, or respectively irrational.
We call the structure $(\lambda, \alpha)
$  \textbf{\emph{scale integral}}, if $c \alpha$ is integral for some $0 \neq c \in \mathbb{R}$.  
Otherwise we call the structure \textbf{\emph{scale irrational}}. 
If $\mathcal{V} $ is non-zero at each point of $M$, 
in particular is a smooth 2-distribution, then such
a structure is called \textbf{\emph{first kind}}.
If $\omega$ is an exact lcs form then we call $\omega
$  integral, rational, irrational, first kind if the exists $\lambda,
\alpha $ s.t. $\omega = d _{\alpha } \lambda $ and
$(\lambda, \alpha ) $ is integral, respectively irrational,
respectively first kind. Similarly define, scale integral,
scale irrational $\omega$. 
\end{definition} 

\begin{definition}\label{def:morphisms} A 
\textbf{\emph{conformal symplectomorphism}} 
of lcs manifolds $\phi: (M _{1}, \omega _{1}) \to (M _{2}, \omega _{2}) $ is a diffeomorphism $\phi$ s.t. $\phi ^{*} 
\omega _{2} = e ^{f} \omega _{1}$, for some $f$.  Note that in this case we have
an induced relation (when $M$ has dimension at least 4): $$\phi ^{*} \alpha _{1} = \alpha _{0}
+ df, $$ where $\alpha  _{1}$ is the Lee form of $\omega
_{1}$ and $\alpha _{0}$ is the Lee form of $\omega _{0}$.
If $f=0$ we call $\phi
$ a \textbf{\emph{symplectomorphism}}. 
A (conformal) symplectomorphism of exact lcs structures
$(\lambda _{1}, \alpha _{1}) $, $(\lambda _{2}, \alpha _{2})
$ on $M _{1}$ respectively $M _{2}$ is a (conformal) symplectomorphism of the corresponding
$lcs$ 2-forms. If a diffeomorphism $\phi: M _{1} \to M _{2}
$ satisfies $\phi ^{*} \lambda _{2} = \lambda _{1}$ and
$\phi ^{*} \alpha _{2} = \alpha _{1}$  we call it an
\textbf{\emph{isomorphism}} of the exact lcs structures.
This is analogous to a strict contactomorphism of contact
manifolds.
\end{definition}
To summarize, with the above notions we have the following
basic points whose proof is left to the reader:
\begin{enumerate}
	\item An isomorphism of exact lcs structures
$(\lambda _{1}, \alpha _{1}) $, $(\lambda _{2}, \alpha _{2})
$ preserves the first kind condition, and moreover preserves
the corresponding vanishing distributions.
\item A symplectomorphism of lcs forms preserves
the first kind condition.  \label{item:preserves}
\item A (conformal) symplectomorphism of exact lcs
structures generally does not preserve the first kind
condition.  (Contrast with \ref{item:preserves}.) 
\item A (conformal) symplectomorphism of first kind lcs
structures generally does not preserve the vanishing
distributions. (Similar to a contactomorphism not preserving Reeb distributions.) 
\item A conformal symplectomorphism of lcs forms and exact
lcs structures preserves the rationality, integrality, scale
integrality conditions.
\end{enumerate}
\begin{remark}
   \label{remark:}
We say that $\omega _{0} $ is conformally equivalent to
$\omega _{1}$ if $\omega _{1}= e ^{f} \omega _{0}$, i.e. the identity map is a conformal symplectomorphism $id: (M,
\omega _{0}) \to (M, \omega _{1}) $.
It is important to note that for us the form $\omega$ is the
structure not its conformal equivalence class, as for some
authors. In other words conformally equivalent structures on a given manifold determine distinct but isomorphic objects of the category, whose objects are lcs manifolds 
and morphisms conformal symplectomorphisms.  
\end{remark}


\begin{example} \label{example:mappingtorus}
One example of an $\lcs$ structure of the 
first kind is a mapping torus of a strict
contactomorphism, see Banyaga~\cite{cite_BanyagaConformal}.
The mapping tori $M _{\phi,c}$ of a strict contactomorphism
$\phi$ of $(C, \lambda )$ fiber over $S
^{1}$,  $$\pi _{}: C \hookrightarrow  M _{\phi,c} \to
S ^{1},$$  with Lee form of the type $\alpha =c\pi ^{*}(d
\theta)$, for some $0 \neq c \in \mathbb{R} ^{} $. In particular, these are scale integral first kind lcs structures.
\end{example}
Moreover we have:
\begin{theorem} [Only reformulating
Bazzoni-Marrero~\cite{cite_BazzoniFirstKind}] \label{thm:firstkindtorus}
A first kind lcs structure $(\lambda, \alpha)$ on a closed
manifold $M$ is isomorphic to the mapping torus of a strict
contactomorphism if and only if it is scale integral.
\end{theorem}
The (scaled) integrality condition is of course necessary
since the Lee form of a mapping torus of a strict
contactomorphism will have this property.  Thus we may
understand  scale irrational
first kind lcs structures as first (and rather dramatic)
departures from the world of contact manifolds into a brave
new lcs world.
\begin{remark} \label{remark:NonTorus}
Note that scale irrational first kind structures certainly exist.
A simple example is given by taking $\lambda, \alpha  $ to
be closed scale irrational 1-forms on $T ^{2}$ with transverse
kernels. Then $\omega = \lambda \wedge \alpha $ is a
scale irrational first kind structure on $T ^{2}$. In particular $(\lambda, \alpha ) $   cannot be a mapping torus of a strict contactomorphism even up to a conformal symplectomorphism.
In general, on a closed manifold we may always perturb
a (first kind) scale integral lcs structure
to a (first kind) scale irrational one.  The examples of the present paper deal with deformations of this sort.
\end{remark}

\subsection{Reeb 2-curves} \label{sec:highergenus}
\begin{definition}\label{def:canonicalDistribution}
Let $(M,\lambda, \alpha)$ be an exact $\lcs$ 
structure and $\omega= d _{\alpha} \lambda$.
Define $X _{\lambda}$ by $\omega (X _{\lambda}, 
\cdot) = \lambda$ and $X _{\alpha} $ by $\omega (X 
_{\alpha}, \cdot) = \alpha$. Let $\mathcal{D} $ 
denote the (generalized) distribution spanned by 
$X _{\alpha }, X _{\lambda }$, meaning $\mathcal{D} 
(p) := \spann (X _{\alpha} (p)  , X _{\lambda 
}(p)) $.  This will be 
called the \textbf{\emph{canonical distribution}}.
\end{definition}
The (generalized) distribution $\mathcal{D} $ is one analogue 
for exact lcs manifolds of the Reeb distribution on 
contact manifolds. A Reeb 2-curve, as defined 
ahead, will be a certain kind of singular leaf of 
$\mathcal{D} $, and so is a kind of 2-dimensional 
analogue of a Reeb orbit.
\begin{example} \label{example:} The simplest example of
a Reeb 2-curve in an exact lcs $(M, \lambda, \alpha ) $, in the case $\mathcal{D} $ is a true 2-dimensional distribution (for example if $(\lambda, \alpha ) $ is first kind), is a closed immersed surface $u: \Sigma \to M$
tangent to $\mathcal{D} $. However, it will be necessary to
consider more generalized curves.
\end{example}

\begin {definition} \label{def:Reeb2curve}
Let $\Sigma$ be a closed nodal Riemann surface (the set 
of nodes can be empty). Let $u: \Sigma \to M$ be a smooth map and let $\widetilde{u}: \widetilde{\Sigma} \to M $ be its normalization 
(see Definition \ref{def:normalization}).
We say that $u$ is a \textbf{\emph{Reeb 2-curve}} 
in $(M,\lambda,\alpha)$, if the following is 
satisfied:
\begin{enumerate}
   \item For each $z \in \widetilde{\Sigma}$,
   $\widetilde{u}  _{*}  (T _{z} 
   \widetilde{\Sigma }  ) = \mathcal{D} 
   (\widetilde{u} (z) ) $, whenever $d \widetilde{u}  (z):
	 T _{z} \Sigma  \to T _{\widetilde{u} (z) }M $ is
	 non-zero, and $\dim \mathcal{D} ({\widetilde{u} (z) }) 
   =2 $. 
   \item $0 \neq [u ^{*} \alpha] \in H ^{1} 
   (\Sigma, \mathbb{R} ^{} ) $.  
\end{enumerate}
\end{definition}
It is tempting to conjecture that every closed exact lcs
manifold has a Reeb 2-curve, in analogy to the Weinstein
conjecture. However this is false:
\begin{theorem} \label{theorem_counterexample}
Let $T ^{2}, g _{st}$ be the 2-torus with its standard flat metric. Let $M _{\widetilde{\phi} ,1}$ be the mapping torus of the unit contangent bundle of $T ^{2}$, with $\widetilde{\phi
} $ corresponding to an isometry $\phi: (T ^{2}, g _{st}) \to
(T ^{2}, g _{st})$, which does not fix the image of any closed
geodesic (an irrational rotation in both coordinates).  Then $M _{\phi,1}$ has no Reeb 2-curves. 
\end{theorem}
Nevertheless, note that for the counterexample above, the
conformal symplectic Weinstein conjecture as described in
the following section readily holds, by Proposition
\ref{lemma_MappingTorusReeb1curve}.

\section{Results on Reeb 2-curves and a conformal symplectic
Weinstein conjecture} \label{sec:Results on Reeb 2-curves}
\begin{definition}\label{def:LM}
  Define \textbf{\emph{the set $\mathcal{L} (M)$  of exact $\lcs$ 
  structures on $M$}},  to be:
  \begin{equation*}
     \mathcal{L} (M) = \{ (\beta, \gamma) \in 
     \Omega ^{1} (M) \times  \Omega ^{1} (M) 
     \,|\, \text{$\gamma$ is closed, $d _{\gamma}  \beta$ is 
     non-degenerate} \}. 
  \end{equation*}
  Define $\mathcal{F} (M) \subset \mathcal{L} (M) 
  $ to be subset of (possibly irrational) first kind lcs structures.
  \end{definition}
In what follows we use the following $C ^{\infty}$ 
metric on $\mathcal{L} (M) $.  For $(\lambda _{1}, \alpha _{1}), (\lambda _{2}, \alpha _{2}) \in \mathcal{L} (M)$ define:
\begin{equation} \label{eq:dk}
   d _{{\infty}}  ((\lambda _{1}, \alpha _{1}), (\lambda
	 _{2}, \alpha _{2})) = d _{C ^{\infty }} (\lambda _{1},
	 \lambda _{2}  )  + d _{C ^{\infty}} (\alpha _{1}, \alpha _{2}  ),
\end{equation}
where $d _{C ^{\infty}} $ on the right side is the usual $C
^{\infty} $ metric. 

The following theorems  are proved in Section 
\ref{sec:Proofs of theorems on Conformal symplectic 
Weinstein conjecture}, based on the theory 
of elliptic  pseudo-holomorphic curves in $M$. We can use $C
^{k}$  metrics  for a certain $k$, instead of $C ^{\infty
}$,
however we cannot make $k=0$ (at least not obviously), and
the extra complexity of working with $C ^{k}$ metrics, is
better left for later developments.
\begin{theorem} \label{thm:C0Weinstein} 
Let $(C, \lambda) $ be a closed contact manifold, satisfying one of
the following conditions:
\begin{enumerate}
	\item $(C, \lambda ) $  has at least one non-degenerate Reeb orbit.
	\item $i (N,R ^{\lambda}, \beta) 
\neq 0$ where the latter is the Fuller index of 
some open compact subset of the orbit space: $N \subset \mathcal{O} (R ^{\lambda}, \beta ) $,  see Appendix \ref{appendix:Fuller}.
\end{enumerate}
Then we have the following: 
\begin{enumerate}
   \item Then for some $d _{\infty} $ neighborhood $U$ of the $\lcs$-fication 
$(\lambda, \alpha)$ of the space $\mathcal{F} (M=C 
\times S ^{1}) $, every element of $U$ admits a Reeb
2-curve.
\item  For any $(\lambda', \alpha' ) \in U$, the corresponding Reeb 2-curve $u: \Sigma \to M$ can be assumed to be \textbf{\emph{elliptic}} meaning that $\Sigma $ is elliptic (more specifically: a nodal, topological genus 1,  
closed, connected Riemann surface).
\item $u$ can also be assumed to be $\alpha$-charge 1 (see 
Definition \ref{def:charge}).
\item  If $M$  has dimension 4 then $u$  
can be assumed to be embedded and normal 
(the set of nodes is empty). And so in particular, such a $u$ 
represents a closed, $(\omega = d _{\alpha } \lambda)
$-symplectic torus hypersurface.
\end{enumerate}
\end{theorem}

\subsection{Reeb 1-curves} \label{sec:Reeb 1-curves}
We have stated some basic new Reeb dynamics phenomena in the
introduction. We now discuss an application of a different
character.

\begin{definition}\label{def:Reeb1Curve}
A smooth map $o: S ^{1} \to M$ is a 
\textbf{\emph{Reeb 1-curve}} in an exact lcs manifold
$(M,\lambda,\alpha)$, if  $$\forall t \in  S ^{1}: 
(\lambda (o' (t)) >0) \land (o' (t)  \in 
\mathcal{D}).$$ 
\end {definition} 
The following is proved in Section \ref{sec:Proofs of theorems on Conformal symplectic Weinstein conjecture}.
\begin{definition} \label{def:ReebCondition}
	We say that an exact lcs manifold $(M, \lambda, \alpha
	) $ satisfies the \textbf{\emph{Reeb condition}} if: $$\lambda (X _{\alpha }) >0.$$
\end{definition}
\begin{theorem} \label{thm:Reeb1curves} Suppose that $(M, 
   \lambda, \alpha ) $ is an exact lcs 
   manifold satisfying the Reeb condition.  If $(M, 
   \lambda, \alpha )$ has an immersed Reeb 2-curve 
   then it also has a Reeb 1-curve. Furthermore, if it has an 
   immersed elliptic Reeb 2-curve, then this Reeb 2-curve is 
   normal. 
\end{theorem}

We have an immediate corollary of Theorem 
\ref{thm:C0Weinstein} and Theorem 
\ref{thm:Reeb1curves}.
\begin{corollary} \label{corol:reeb1curve} 
Let $\lambda$ be a contact form, on closed 
3-manifold $C$, with at least one non-degenerate 
Reeb orbit, or more generally satisfying $i (N,R ^{\lambda},
\beta) \neq 0$ for some open compact $N$ as previously.
Then there is a 
$d _{\infty} $ neighborhood $U$ of the $\lcs$-fication 
$(\lambda, \alpha)$ in the space $\mathcal{F} (C 
\times S ^{1}) $, s.t. for each $(\lambda', \alpha 
') \in U$ there is a Reeb 1-curve.
\end{corollary}
\begin{corollary} \label{lemma:conformalWeinsteinSurface}
Every closed exact lcs surface satisfying the Reeb condition
has a Reeb 1-curve.
\end{corollary}
\begin{lemma}
   \label{lemma:alphaclass} Let $(M, \lambda, \alpha) $  be 
   an exact lcs manifold with $M$ closed then $0 
   \neq [\alpha] \in H ^{1} 
   (M, \mathbb{R} ^{} ) $.  
\end{lemma}
\begin{proof} 
Suppose by contradiction that $\alpha$ is exact and let $g$ be its primitive.   Then computing we get: $d _{\alpha} \lambda = 
   \frac{1}{f} d (f \lambda) $ with $f = e  
   ^{g}$. Consequently, $d (f \lambda) $ is 
   non-degenerate on $M$ which contradicts Stokes 
   theorem. 
\end{proof}

\begin{proof} [Proof of Corollary \ref{lemma:conformalWeinsteinSurface}]
This follows by Theorem \ref{thm:Reeb1curves}. As in this
case by the lemma above the identity map $X \to X$ is an
immersed Reeb 2-curve, by Lemma \ref{lemma:alphaclass}.
\end{proof}

\begin{proposition} \label{lemma_MappingTorusReeb1curve}
Assume the Weinstein conjecture, then the mapping torus $M
_{\phi }$ of a strict contactomorphism $\phi: (C, \lambda
) \to (C, \lambda )$, where $C$ is closed, has a Reeb 1-curve.
\end{proposition}
\begin{proof} [Proof] Immediate from definitions.
\end{proof}
Inspired by the above considerations we conjecture:
\begin{conjecture} \label{conj:conformalWeinstein}
Suppose that $(M, \lambda, \alpha ) $ is a closed exact lcs 
manifold of dimension $4$ satisfying the Reeb
condition, then it has a Reeb 1-curve.
\end{conjecture}

\begin{proposition} \label{prop:Implies}
The analogue of Conjecture \ref{conj:conformalWeinstein}, in
all dimensions, implies the 
Weinstein conjecture: every closed contact 
manifold $(C,\lambda)$ has a closed Reeb orbit. 
\end{proposition}
\begin{proof} [Proof of Proposition \ref{prop:Implies}]
Let $(M= C \times S ^{1}, \lambda, \alpha)$ be the
$\lcs$-fication of a closed contact manifold $(C, \lambda)
$. Then it satisfies the Reeb condition.
Suppose that $o: S ^{1} \to M$ is a Reeb 1-curve.
Then $ \forall t \in [0,1]: \lambda (\dot o (t)) >0 $  and 
$o$ is tangent to $\mathcal{V} _{\lambda} = \mathcal{D}$.
Consequently, $pr _{C} \circ o$ is tangent to
$\ker d \lambda $, and $ \forall \tau \in [0,1]: \lambda
(({pr _{C} \circ o})' (\tau)) >0 $.   It follows that $pr _{C}
\circ o$ is a Reeb orbit of $(C, \lambda )$ up to parametrization.
\end{proof}

\section{$J$-holomorphic curves in lcs manifolds and the
definition of the invariant $\operatorname {GWF}$ } 
\label{sec:gromov_witten_theory_of_the_lcs_c_times_s_1_} 
Let $M,J$ be an almost complex manifold and $\Sigma,j$
a Riemann surface. Recall that a map
$u: \Sigma \to M$ is
said to be \emph{$J$-holomorphic}  if $du \circ j = J \circ du$.
\begin{notation}
   \label{notation:jcurve} We will often say 
   \textbf{\emph{$J$-curve}}  in place of $J$-holomorphic curve.
\end{notation}
First kind $\lcs$ manifolds give immediate examples of almost complex manifolds where the $L ^{2} $ $\energy$ functional is unbounded on the moduli spaces of fixed class $J$-curves, as well as where null-homologous $J$-curves can be non-constant. We are going to see this shortly after developing a more general theory.
\begin{definition} \label{def:admissible} Let 
$(M,\lambda, \alpha)$ be an exact lcs manifold, 
satisfying the \textbf{\emph{Reeb condition}}:  $\omega (X _{\lambda}, X 
_{\alpha}) = \lambda (X _{\alpha})  > 0$, where $\omega
= d _{\alpha} \lambda  
$. In this case, $\mathcal{D} $ is a 2-dimensional 
distribution, and we say that an $\omega$-compatible $J$ 
is \textbf{\emph{$(\lambda,\alpha)$-admissible}} or
$\omega$-admissible (when $\lambda,\alpha $ are implicit) if:
\begin{itemize}
\item  $J$ preserves  the canonical distribution 
$\mathcal{D} $ and preserves the $\omega$-orthogonal complement $\mathcal{D}  ^{\perp}$ of 
$\mathcal{D}$. That is 
$J ({V}) \subset \mathcal{D}   $ and 
$J(\mathcal{D}  ^{\perp}) \subset \mathcal{D}  
^{\perp}$. 
\item $d\lambda$ tames $J$ on $\mathcal{D}  ^{\perp}$.
\end{itemize}

Admissible $J$ exist by classical symplectic 
geometry, and the space of such $J$ is 
contractible see
~\cite{cite_McDuffSalamonIntroductiontosymplectictopology}.  We call $(\lambda, \alpha, J)$ as above a 
\textbf{\emph{tamed exact $\lcs$ structure}}, and  
$(\omega,J)$ is called a tamed exact $\lcs$ 
structure if $\omega = d_{\alpha}\lambda $, for 
$(\lambda, \alpha, J)$ a tamed exact $\lcs$ 
structure. In this case $(M, \omega,J) $, 
$(M, \lambda, \alpha, J) $   will be called a 
\textbf{\emph{tamed exact $\lcs$ manifold}}.
\end{definition}
\begin{example}
   \label{example:firstKindtamed} If 
   $(M,\lambda,\alpha)$ is first kind then by an elementary computation $\omega (X _{\lambda}, X 
_{\alpha}) =1$ everywhere. In particular, we may find 
a $J$ such that $(\lambda, \alpha, J)$ is a tamed 
exact lcs structure, and the space of such $J$ is 
contractible. We will call $(M, \lambda, \alpha, 
J) $ a \textbf{\emph{tamed first kind}}  lcs 
manifold.
\end{example}
\begin{lemma} \label{lemma:calibrated} Let $(M, \lambda, \alpha, J) $
 be a tamed first kind lcs manifold.  Then given a 
 smooth $u: \Sigma \to M$, where $\Sigma$ is a 
 closed (nodal) Riemann surface, $u$ is 
 $J$-holomorphic only if $$\image d {\widetilde{u} }  
 (z) \subset \mathcal{V} _{\lambda} ( 
 {\widetilde{u} }  (z)) $$ 
 for all $z \in \widetilde{\Sigma } $, where  $\widetilde{u}: \widetilde{\Sigma } \to M $ is the 
normalization  of $u$ (see Definition \ref{def:normalization}).
 In particular $\widetilde{u}  ^{*} 
 d\lambda =0$.  
\end{lemma}
\begin{proof}  As previously observed, by the 
first kind condition, $\mathcal{V} _{\lambda}$ is 
the span of $X _{\lambda}, X _{\alpha}$  and hence 
$$V:=\mathcal{V} _{\lambda } = \mathcal{D} 
_{\lambda}. $$ Let $u$ be $J$-holomorphic, so that 
$\widetilde{u} $ is $J$-holomorphic (by 
definition of a $J$-holomorphic nodal map). 
We have $$\int 
_{\Sigma} \widetilde{u}  ^{*} d \lambda =  0 $$ by Stokes 
theorem. Let $proj (p): T _{p} M \to 
V ^{\perp}(p)  $ be the projection induced by the 
splitting $TM = V \oplus V ^{\perp}$. 

Suppose that for  some $z \in \widetilde{\Sigma }  $, $proj
\circ d \widetilde{u}  (z) \neq 0 $. By the conditions:
\begin{itemize}
	\item $J$ is tamed by $d\lambda$ on $V ^{\perp} $. 
	\item $d \lambda $ vanishes on $V$.
	\item $J$ preserves the splitting $TM = V \oplus
	V ^{\perp} $.
\end{itemize}
we have $\int _{\widetilde{\Sigma }  } \widetilde{u}  ^{*} d 
\lambda >  0$, a contradiction. Thus, $$\forall z \in 
\widetilde{\Sigma }: proj \circ d \widetilde{u}  (z) = 0, $$ so 
$$\forall z \in \widetilde{\Sigma } : \image d 
\widetilde{u}  (z) \subset \mathcal{V} _{\lambda} 
(\widetilde{u}  (z)).  $$

\end{proof}
\begin{example} \label{section:lcsfication}
Let $(C \times S ^{1}, \lambda, \alpha) $ be the 
lcs-fication of a contact manifold $(C, \lambda) 
$. In this case $$X _{\alpha} = (R ^{\lambda},  
0),$$ where $R ^{\lambda}$ is the Reeb vector 
field and $$X _{\lambda} = (0, \frac{d}{d \theta}) 
$$ is the vector field generating the natural action 
of $S ^{1} $ on $C \times S ^{1}  $.  

If we denote by $\xi \subset T (C \times S ^{1} )$ the 
distribution $\xi (p) = \ker \lambda (p)$, then in 
this case $\xi = V ^{\perp}$ in the notation 
above. We then take $J$ to be an almost complex structure on 
$\xi$, which is $S ^{1} $ invariant, and 
compatible with $d\lambda$.  The latter means that 
$$g _{J} (\cdot, \cdot):= d \lambda| _{\xi} 
(\cdot, J \cdot)$$ is a $J$ invariant Riemannian metric on the distribution $\xi$. 

There is an induced almost complex structure $J ^{\lambda}$ on $C \times S ^{1} $, which is $S ^{1}$-invariant,  coincides with $J$ on $\xi$ and which satisfies: 
\begin{equation*}
J ^{\lambda} (X _{\alpha}) = X _{\lambda}.
\end{equation*}
Then $(C \times S ^{1}, \lambda, 
\alpha, J ^{\lambda} )$ is a tamed integral first kind $\lcs$ 
manifold.
\end{example}

\subsection {Charged elliptic curves in an lcs manifold} 
We now study moduli spaces of elliptic curves in a lcs
manifold, constrained to have a certain charge. \footnote {The name charge is inspired by the notion of charge in Oh-Wang
\cite{cite_OhWang}, in the context of contact
instantons. However, the respective notions are not obviously
related.}  In the present context,
one reason for the introduction of ``charge'' is that 
it is now possible for non-constant holomorphic 
curves to be null-homologous, so we need 
additional control.  Here is a simple example: take 
$S ^{3} \times S ^{1}  $ with $J=J ^{\lambda} $, 
for the $\lambda$ the standard contact form, then 
all the Reeb holomorphic tori (as defined further 
below) are null-homologous. 

Let $\Sigma$ be a complex torus with a chosen 
marked point $z \in \Sigma $, i.e. an elliptic 
curve over $\mathbb{C} $.  An isomorphism $\phi: 
(\Sigma _{1}, z _{1} ) \to (\Sigma _{2}, z _{2} )$ 
is a biholomorphism s.t. $\phi (z _{1} ) = z _{2} 
$. The set of isomorphism classes forms a smooth 
orbifold $M _{1,1} $. This has a natural 
compactification - the Deligne-Mumford 
compactification $\overline{M} _{1,1}  $, by 
adding a point at infinity, corresponding to a 
nodal genus 1 curve with one node. 

The notion of charge can be defined in a general setting.
\begin{definition}\label{def:charge}
Let $M$ be a manifold endowed with a closed integral 1-form $\alpha$.
Let $u: T^{2} \to M $ be a continuous map. Let $\gamma,
\rho: S ^{1} \to T ^{2}$ represent generators of $H _{1} (T ^{2}, \mathbb{Z})
$, with $\gamma \cdot \rho =1$, where $\cdot$ is the
intersection pairing with respect to the standard complex
orientation on $T ^{2}$.
Suppose in addition:
\begin{equation*}
\langle \gamma, u ^{*} {\alpha} \rangle =0, \quad
\langle \rho, u ^{*} {\alpha} \rangle \neq 0,
\end{equation*} 
where
$\langle ,  \rangle $ is the natural pairing of homology and
cohomology.
Then we call 
\begin{equation*}
n = |\langle \rho, u ^{*} {\alpha} \rangle | \in
\mathbb{N}_{>0},
\end{equation*}
the \textbf{\emph{$\alpha $-charge of $u$}}, or just the
charge of
$u$ when $\alpha $ is implicit. Suppose furthermore that
$\langle \rho, u ^{*} {\alpha} \rangle > 0,
$
then the class $u _{*} (\gamma)  \in \pi _{1}  (M)$ will be
called the \textbf{\emph{$\pi$-class of $u$}}, for $\pi _{1}
(M)$ the set of free homotopy classes of loops as before.
\end{definition} 
It is easy to see that charge is always defined and is
independent of choices above.  We may extend the definition
of charge to curves $u: \Sigma \to M$, with $\Sigma $ a nodal
elliptic curve, as follows.  
If $\rho: S ^{1} \to \Sigma $  represents the generator of
$H _{1} (\Sigma, \mathbb{Z} ) $ then define 
the charge of $u$ to be $|\langle \rho, u ^{*}\alpha   \rangle |$.
Obviously the charge condition is preserved under Gromov
convergence of stable maps. But it is not preserved in
homology, so that charge is \emph{not} a functional $H _{2} (M,
\mathbb{Z} ) \to \mathbb{N}  $.
\begin{definition}\label{definition_chargeclass}
By the above, associated to a continuous map $u: \Sigma \to M$ with $\Sigma $ an elliptic curve, and non-zero $\alpha $-charge, we have
a triple $(A, \beta, n) \in H ^{2} (M,
\mathbb{Z} ) \times \pi  _{1} (M) \times \mathbb{N} _{>0} $,
corresponding to the homology class, the $\pi _{} $-class,
and the $\alpha $-charge. This triple will be called the
\textbf{\emph{charge class}} of $u$.

\end{definition}

Let $(M, J)$ be an almost complex manifold and $\alpha
$ a closed integral 1-form on $M$ non vanishing in cohomology, then
we call $(M, J, \alpha) $ a \textbf{\emph{Lee manifold}}.
Suppose for the moment that there are no non-constant $J$-holomorphic maps $(S ^{2},j) \to 
(M,J)$ (otherwise we need stable maps), then for $n \geq 1$ we define: 
$$\overline{\mathcal{M}} ^{n} _{1,1}  (J, A, \beta)$$ as the set of equivalence classes of tuples $(u, S)$, for $S= (\Sigma, z) 
$ a possibly nodal elliptic curve and $u: \Sigma \to M$ 
a charge class $(A,\beta,n)$, $J$-holomorphic map. 
The equivalence relation is $(u _{1}, S _{1}  ) \sim (u _{2}, S _{2}  )$ if there is an isomorphism $\phi: S _{1} \to S _{2}  $ s.t. $u _{2} \circ \phi = u _{1} $.
It is not hard to see that such an isomorphism of preserves
the charge class, so that $\overline{\mathcal{M}} ^{n}
_{1,1}  (J, A, \beta)$ is well defined.

Also note that the expected dimension of $\overline{\mathcal{M}} 
_{1,1} ^1   ({J} ^{\lambda},
A, \beta)$ is 0. It is given by the
Fredholm index of the operator \eqref{eq:fullD} 
which is 2, minus the dimension of the 
reparametrization group (for non-nodal curves) 
which is 2. That is given an elliptic curve $S = 
(\Sigma,  z) $,  let $\mathcal{G} (\Sigma)$ be the 
2-dimensional group of biholomorphisms $\phi$ of 
$\Sigma$. Then given a $J$-holomorphic map $u: 
\Sigma \to M$, $(\Sigma,z,u)$ is equivalent to $(\Sigma, \phi(z), u 
\circ \phi)$ in $\overline{\mathcal{M}} _{1,1} ^1  ({J}
^{\lambda}, A, \beta)$, for $\phi \in \mathcal{G} (\Sigma)$.


By slight abuse we may just denote such an equivalence class 
above simply by $u$, so we may write $u \in
\overline{\mathcal{M}} ^{n} _{1,1}  (J, A, \beta) $, with $S$ implicit.
\subsection{Reeb holomorphic tori in $(C \times S ^{1},
J ^{\lambda}) $} \label{sec:Reeb holomorphic tori simple}
In this section we discuss an important example.
Let $(C, \lambda) $ be a contact manifold and let $\alpha $ and $J ^{\lambda }$ be as in Example \ref{section:lcsfication}.
So that in particular we get a Lee manifold $(C \times S ^{1}, J ^{\lambda}, \alpha) $. 

In this case we have one natural type of charge 1 $J
^{\lambda }$-holomorphic tori in $M = C \times S ^{1} $.
Let $o$ be a period $c$, closed Reeb orbit $o$ of $R
^{\lambda} $, and let $\beta $ it's class in $\pi _{1} (C)
\subset \pi _{1} (M)$. A \textbf{\emph{Reeb torus}} $u _{o} $ for $o$ is the 
map 
\begin{align*}
u _{o}: (S ^{1} \times S ^{1} = T ^{2}) \to C \times S ^{1} \\
u_o (s, t)  = (o  (s), t).
\end{align*}
A Reeb torus is $J ^{\lambda} $-holomorphic for a uniquely determined holomorphic structure $j$ on $T ^{2} $ defined by:
$$j
(\frac{\partial}{\partial s}) = c \frac{\partial} {\partial t}. $$



\subsection{Definition of the invariant $\operatorname
{GWF}$} \label{sec:Definition of GWF} 
Let $X$ be a manifold. For $g$ a taut metric on $X$, let $\lambda _{g}$ be the Liouville 1-form on the
unit cotangent bundle $C$ of $X$. If $\phi$ is an isometry
of $g$ then there is a strict contactomorphism
$\widetilde{\phi } $ of $(C, \lambda _{g})$, and this gives
the ``mapping torus'' lcs manifold $(M _{\widetilde{\phi },
1}, \lambda _{\widetilde{\phi} }, \alpha ) $ as described in
Section \ref{sec:Mapping tori and Reeb 2-curves}.

If $ \beta \in \pi _{1} ^{inc}  (X)$, let $\widetilde{\beta}
\in \pi _{1}  (C)
$ denote the lift of the class, defined by representing $\beta $ by a unit speed closed geodesic
$o$, taking the canonical lift $\widetilde{o}$ to a closed
Reeb orbit, and setting $\widetilde{\beta} = [\widetilde{o}
]$. Given $n \geq 1 $, suppose that 
\begin{equation} \label{eq:ncharephi}
\widetilde{\phi} ^{n} _{*}
(\widetilde{\beta}) = \widetilde{\beta }.
\end{equation}
Then, as explained  in Section \ref{sec:Mapping tori and
Reeb 2-curves}, this naturally induces a map $u ^{n}: T ^{2} \to M$ well
defined up to homotopy,  whose class in homology is denoted
by $A ^{n} _{\widetilde{\beta } } \in H _{2} (M, \mathbb{Z}
)$.  The $\alpha $-charge of $u ^{n}$ is $n$, and its $\pi _{}
$-class is $\widetilde{\beta } $.

By the tautness assumption on $g$ the space $\mathcal{O} (R ^{\lambda _{g}}, \widetilde{\beta}
)$ is compact. 
We then get that
$\overline{\mathcal{M}} ^n _{1,1}  (J ^{\lambda _{\phi }}, A ^{n}
_{\widetilde{\beta}}, \widetilde{\beta } )$ is compact and has expected dimension
0 by the Proposition
\ref{prop:Topembedding}. 
We then define $$\operatorname
{GWF} (g, \phi, \beta, n):= \begin{cases}
	0, &\text{ if \eqref{eq:ncharephi} is not satisfied}\\
	GW ^{n} _{1,1}
(J ^{\lambda _{\phi}}, A ^{n}_{\widetilde{\beta}
}, \widetilde{\beta}) ([\overline {M} _{1, 1}] \otimes [C \times S ^{1} ]), &\text{ otherwise}
\end{cases}, $$
where the Gromov-Witten invariant on the right side is as in \eqref{eq:functionals2}, of the following section.
Although we take here a specific almost complex structure $J
^{\phi }$, using Proposition \ref{prop:Topembedding} and
Lemma \ref{prop:invariance1} we may readily
deduce that any $(\lambda _{\phi}, \alpha )$-admissible
almost complex structure gives the same value for the
invariant. 

Also note that when $\phi =id$
\begin{equation} \label{eq_GWF=F}
\forall n \in \mathbb{N}: {GWF} (g, id, \beta, n) = F (g,
\beta ),
\end{equation}
where the latter is the invariant studied in
~\cite{cite_SavelyevFuller}. This readily follows by Theorem
\ref{thm:GWFullerMain}.
\section {Elements of Gromov-Witten theory of an almost
complex manifold} \label{sec:elements}
Suppose that $(M,J)$ is an almost complex manifold (possibly
non-compact), where
the almost complex structures $J$ are assumed throughout the
paper to be $C ^{\infty} $. 
Let 
   $N \subset
\overline{\mathcal{M}} _{g,k}   (J,
A) $ be an open compact subset with $\energy$ positive on $N$. 
The latter energy condition is only relevant when $A =0$.  
We shall primarily refer in what follows to work of Pardon
 in \cite{cite_PardonAlgebraicApproach}, being more familiar
 to the author.
But we should mention that the latter is a follow up to
a theory that is originally created by Fukaya-Ono
\cite{cite_FukayaOnoArnoldandGW}, and later expanded with
Oh-Ohta
\cite{cite_FukayaLagrangianIntersectionFloertheoryAnomalyandObstructionIandII}.

 The construction in \cite{cite_PardonAlgebraicApproach} of 
 an implicit atlas,  on the moduli space $\mathcal{M}$  of 
 $J$-curves in a symplectic manifold, only needs a neighborhood 
 of $\mathcal{M}$ in the space of all curves. So for an 
 \emph{open} compact component $N$ as above,  we have a well
 defined natural implicit atlas, (or a Kuranishi
 structure in the setup of
 \cite{cite_FukayaOnoArnoldandGW}).  And so such an $N$ will
 have a virtual fundamental class in the sense of
 ~\cite{cite_PardonAlgebraicApproach}. This understanding will be used in other parts of the paper, following  Pardon for the explicit setup. 
 
We may thus define functionals:
\begin{equation} \label{eq:functionals2nonmain}
GW _{g,n}  (N,J, A): H_* (\overline{M} _{g,n}) \otimes H _{*} (M )  \to
   \mathbb{Q}.
\end{equation}

In our more specific context we must in addition restrict
the charge, which is defined at the moment for genus 1 curves.
So supposing $(M, J, \alpha) $ is a Lee manifold we may likewise define functionals:
\begin{equation} \label{eq:functionals2}
GW _{1,1} ^{k}  (N,J, A, \beta): H_* (\overline{M} _{1,1}) \otimes H _{*} (M)  \to \mathbb{Q},
\end{equation}
meaning that we restrict the count to charge class
$(A, \beta, k)$ curves, with $N \subset
\overline{\mathcal{M}} ^{k} _{1,1}   (J, A, \beta ) $,
an open compact subset. If $N$ is not
specified it is understood to be the whole moduli space (if it is
known to be compact).

We now study how functionals depend on $N,J$. To avoid
unnecessary generality, we discuss the
case of $GW _{1,1} ^{k}  (N,J, A, \beta )$. Given a Frechet smooth family
$\{J _{t} \}$, $t \in [0,1]$, on $M$, we denote by $\overline{\mathcal{M}} ^{k}_{1,1}
(\{J _{t} \}, A, \beta)$ the space of pairs $(u,t)$, $u \in
\overline{\mathcal{M}} ^{k}_{1,1}(J _{t}, A, \beta )$.

\begin{lemma} \label{prop:invariance1} Let $\{J _{t} \}$, $t \in [0,1]$ be a Frechet smooth family of almost complex structures on $M$. Suppose that $\widetilde{N}$ is an open compact subset of the cobordism moduli space $\overline{\mathcal{M}} _{1,1} ^{k}   (\{J _{t} \},
A, \beta)
$, with $k>0$.  Let $$N _{i} = \widetilde{N} \cap \left(
\overline{\mathcal{M}} _{1,1} ^{k}   (J _{i}, A, \beta )\right),  $$
then $$GW _{1,1} ^{k}  (N _{0}, J _{0}, A ) = GW _{1,1} ^{k}(N _{1},
 J _{1}, A, \beta ).  $$ In particular if $GW ^{k}_{1,1}  (N _{0},
 A, J _{0}, \beta  )
   \neq 0$, there is a $J _{1} $-holomorphic,
	 stable, charge class $(A, \beta, k)$ elliptic curve in $M$.
\end{lemma} 
\begin{proof} [Proof of Lemma \ref{prop:invariance1}]
We may construct exactly as in
\cite{cite_PardonAlgebraicApproach} a natural implicit atlas
on $\widetilde{N} $, with boundary $N _{0}  ^{op} \sqcup
N _{1} $, ($op$ denoting opposite orientation). And so we immediately get 
\begin{equation*}
   GW _{1,1} ^{k}   (N  _{0}, J _{0}, A, \beta  ) = GW _{1,1} ^{k}
	 (N _{1},  J _{1}, A, \beta ).
\end{equation*}
\end{proof}

\begin{remark} \label{remark:closed}
The condition that $k>0$ is a simple way to rule out degenerations to
constant curves, but is not really essential.
In the case the manifold is closed, degenerations of
$J$-holomorphic curves to constant curves are impossible. This can be deduced from energy quantization coming from the general
monotonicty theorem as appearing in Zinger
~\cite[Proposition 3.12]{cite_zinger2017notes}.
This was noted to me by Spencer Cattalani. Even if the
manifold is not compact, given the assumption that
$\widetilde{N} $ itself is compact, we may similarly
preclude such degenerations.
%
\end{remark}

The following generalization of the lemma above will be
useful later. First a  definition. 
\begin{definition} Let $M$ be a smooth manifold.
Denote by $H _{2} ^{inc} (M)$, the set of boundary
incompressible homology classes, defined analogously to
Definition \ref{definition_boundaryincompressible}, 
We say that a Frechet smooth family $\{J _{t}\}$, $t \in
[0,1]$ on a  manifold $M$ has a \textbf{\emph{right holomorphic
sky catastrophe}} in charge class $(A, \beta,k)$ for $A \in
H _{2} ^{inc}(M)$, if
there is an element $u \in \overline{\mathcal{M}} ^{k} _{1
,1} (J _{0}, A, \beta) $, 
which does not belong to any open compact subset of
$\overline{\mathcal{M}} _{1,1} ^{k} (\{J _{t} \}, A, \beta)$. We
say that the sky catastrophe is \textbf{\emph{essential}}
if the same is true for any smooth family $\{J' _{t}\}$ satisfying $J'_0 = J _{0}$   and $J' _{1}
= J _{1}$.
\end{definition}

\begin{lemma} \label{thm:welldefined} 
Let $\{J _{t} \}$, $t \in [0,1]$ be a Frechet smooth family
of almost complex structures on $M$, $A \in H _{2} ^{inc} (M)$ and $k>0$. Suppose that
$\overline{\mathcal{M}} _{1,1} ^{k}   (J _{0}, A, \beta ) $  is
compact, and there is no right holomorphic sky
catastrophe for $\{J _{t}\}$.  Then there is a charge class
$(A, \beta , k)$,  $J _{1} $-holomorphic,
stable, elliptic curve in $M$.
\end{lemma}
\begin{proof} [Proof]
By assumption for each $u \in \overline{\mathcal{M}} _{1,1}
^{k}   (J _{0}, A, \beta )  $ there is an open compact $u \ni
\mathcal{C} _{u}
\subset \overline{\mathcal{M}} _{1,1} ^{k} (\{J _{t} \},
A, \beta )$. Then $\{\mathcal{C} _{u} \cap \overline{\mathcal{M}} _{1,1}
^{k}   (J _{0}, A, \beta ) \} _{u}$ is an open cover of $\overline{\mathcal{M}} _{1,1}
^{k}   (J _{0}, A, \beta )$ and so has a finite sub-cover,
corresponding to a collection $u _{1}, \ldots, u _{n}$.

Set $$\widetilde{N}    = \bigcup _{i \in \{1, \ldots, n\}} \mathcal{C} _{i}.  $$  
Then $\widetilde{N}$ is an open-compact subset of $ \overline{\mathcal{M}} _{1,1} ^{k} (\{J _{t} \},
A)$ s.t. $\widetilde{N} \cap  \overline{\mathcal{M}} _{1,1} ^{k} (J _{0},
A, \beta ) = N_0 := \overline{\mathcal{M}} _{1,1} ^{k} (J
_{0}, A, \beta ) $.
Then the result follows by Lemma \ref{prop:invariance1}.

\end{proof}

We now state a basic technical lemma, following some standard
definitions. 
\begin{definition} An \textbf{\emph{almost symplectic pair}} 
on $M$ is a tuple $(\omega, J)$, where $\omega$ is a 
non-degenerate 2-form on $M$, and $J$ is 
$\omega$-compatible, meaning that $\omega (\cdot, J \cdot)$ 
defines $J$-invariant Riemannian metric, denoted by $g _{J}$
(with $\omega$ implicit). 
\end{definition}

\begin{definition} \label{def:deltaclose} We say that a pair 
of almost symplectic pairs $(\omega _{i}, J _{i}  )$ are 
\textbf{$\delta$-close}, if $\omega _{0}, \omega _{1}$ are $C
^{\infty} $ $\delta$-close, and $J _{0}, J _{1} $ are $C ^{\infty} $ 
$\delta$-close, $i=0,1$. 
\end{definition}

Let $\mathcal{S} (A) $ denote the space of equivalence
classes of all smooth, nodal, stable, charge $k$,  elliptic curves in $M$ in class
$A$, with the standard Gromov topology determined by $g
_{J}$.  That is elements of $\mathcal{S} (A) $ are like
elements of $\overline{\mathcal{M}} ^{k} _{1,1}  (J,A, \beta )$ but
are not required to be $J$-holomorphic. In particular, we
have a continuous function:
\begin{equation*}
e = e _{g _{J}}: \mathcal{S} (A) \to \mathbb{R} _{\geq 0}. 
\end{equation*}

\begin{lemma} \label{lemma:neighborhood}
Let $(\omega, J)$  be an almost symplectic pair on a compact
manifold $M$ and let $N \subset \overline{\mathcal{M}} ^{k}
_{1,1}  (J,A, \beta ) $ be compact and open (as
a subset of $\overline{\mathcal{M}} ^{k}
_{1,1}  (J,A)$). Then there exists an open $U \subset \mathcal{S} (A) $  satisfying: 
\begin{enumerate} 
	\item $e$ is bounded on $\overline{U} $. \label{property:ebounded}
	\item $U \supset N $. \label{property:supset}
	\item $\overline{U}  \cap  \overline{\mathcal{M}} ^{k}
	_{1,1}  (J,A, \beta ) = N$. \label{property:intersection}
\end{enumerate}
\end{lemma}
\begin{proof}
The Gromov topology on $\mathcal{S} (A) $ has a basis
$\mathcal{B}$ satisfying:
\begin{enumerate}
	\item If $V \in \mathcal{B} $ then  $e$ is bounded on
$\overline{V}$.
\item If $U$ is open and $u \in U$, then 
\begin{equation*}
\exists V \in \mathcal{B}: (u \in V) \land (\overline{V}
\subset U). 
\end{equation*}
\end{enumerate}
In the genus 0 case this is contained in the classical
text McDuff-Salamon~\cite[page
140]{cite_McDuffSalamonJholomorphiccurvesandsymplectictopology}.
The basis $\mathcal{B}$  is defined using a collection of
``quasi distance functions'' $\{\rho _{\epsilon}\}
_{\epsilon}$ on the set stable maps. The higher genus case
is likewise well known.

Thus, since $N$ is relatively open, using the properties of
$\mathcal{B} $ above, we may find a collection
$\{V _{\alpha}\} \subset \mathcal{B} $ s.t. 
\begin{itemize}
\item $\{V _{\alpha}\}$ covers $N$.
\item $\overline{V} _{\alpha} \cap \overline{\mathcal{M}}
^{k} _{1,1}  (J,A, \beta ) \subset N.$
\end{itemize}
As $N$ is compact, we have a finite subcover $\{V _{\alpha _{1}}, \ldots, V _{\alpha _{n}}\}$. Set $U:= \cup _{i \in
\{1, \ldots, n\}} V _{\alpha _{i}}$.
Then $U$ satisfies the conclusion of the lemma.
\end{proof}

\begin{lemma} \label{lemma:NearbyEnergy} Let $(M, \omega,
J, \alpha )$ be as above, $N \subset \overline{\mathcal{M}}
^{k} _{1,1} (J,A, \beta ) $ an open compact set, and $U$ as in the lemma above. 
Then there is a $\delta>0$ s.t. whenever $J'$ is $C ^{2}$
$\delta$-close to $J$ if $u \in \overline{\mathcal{M}} ^{k}
_{1,1}  (J',A, \beta ) $ and $u \in \overline{U}  $ then $u \in U$.
\end{lemma}

\begin{proof} Suppose otherwise, then there is a 
sequence $\{J _{k} \}$  $C ^{2}$ converging to 
$J$, and a sequence $\{u _{k} \} \in \overline{U} -U$ of $J _{k} $-holomorphic stable maps. Then by property
\ref{property:ebounded} $e _{g _{J}}$ is bounded on $\{u
_{k} \}$.
Hence, by Gromov compactness, specifically theorems \cite[B.41,
 B.42]{cite_McDuffSalamonJholomorphiccurvesandsymplectictopology}, 
we may find a Gromov convergent subsequence $\{u 
_{k _{j} } \}$ to a $J$-holomorphic stable map $u \in
\overline{U} - U$. 
But by Properties \ref{property:intersection},
\ref{property:supset} of the set $U$,  $$(\overline{U} - U)
\cap \overline{\mathcal{M}} ^{k} _{1,1}  (J,A, \beta ) = \emptyset. $$ So that we obtain a contradiction.
\end{proof}
 
\begin{lemma} \label{lemma:NearbyEnergyDeformation} 
Let $M,
\omega, J, \alpha $ and $N \subset \overline{\mathcal{M}}
^{k} _{1,1}  (J,A, \beta ) $ be as in the previous lemma.
Then there is a $\delta>0$ s.t. the following is satisfied.
Let $(\omega',J')$ be $\delta$-close to $(\omega,J)$, then
there is a continuous in the $C ^{\infty} $ topology family
$\{ J _{t}  \}$, $J_0= J$,  $ J_1=  J'$ s.t. there is an open compact subset 
\begin{equation*}
  \widetilde{N}  \subset \overline{\mathcal{M}} ^{k} _{1,1}
	(\{J _{t}\},A, \beta ), 
 \end{equation*}
satisfying $$\widetilde{N} \cap \overline{\mathcal{M}} ^{k}
_{1,1}  (J,A, \beta ) = N. $$
\end{lemma}
\begin{proof}  First let $\delta$ be as in Lemma \ref{lemma:NearbyEnergy}. We then need:
\begin{lemma} \label{lemma:Ret} Given a $\delta>0$  there is 
a $\delta'>0$ s.t. if $(\omega',J')$ is $\delta'$-near 
$(\omega, J)$ then there is a continuous in the $C ^{\infty} $ 
topology family $\{(\omega _{t}, J _{t}  )\}$ satisfying: 
\begin{itemize}
   \item $(\omega _{t},J _{t}  )$ is $\delta$-close to $(\omega,J)$ for each $t$.
   \item $(\omega _{0}, J _{0}) = (\omega, J) $ and $(\omega 
   _{1}, J _{1}) = (\omega', J') $.
\end{itemize}
\end{lemma}
\begin{proof}  Let $\{g _{t} \} $ be the family 
of metrics on $M$ given by the convex linear 
combination of $g=g _{\omega _{J} } ,g' = g 
_{\omega',J'}  $, $g _{t}= (1-t)  g + t g'$.  Clearly $g
_{t} $ is $C ^{\infty}$ $\delta'$-close to $g _{0} $ for each $t$. 
Likewise, the family of 2 forms $\{\omega _{t} \}$ given by
the convex linear combination of $\omega $, $\omega'$ is
non-degenerate for each $t$ if $\delta'$ was chosen to be
sufficiently small. And each $\omega _{t}$  is $C ^{\infty}$  $\delta'$-close to $\omega _{0} = \omega _{g,J}  $.

    Let $$ret: Met (M) \times \Omega (M)  \to \mathcal{J}
		(M)  $$ be the  ``retraction map'' (it can be understood
		as a retraction followed by projection) as defined in
		\cite [Prop
		2.50]{cite_McDuffSalamonIntroductiontosymplectictopology}, where $Met (M)$ is  space of metrics on $M$, $\Omega (M)$ the space of 2-forms on $M$, and $ \mathcal{J} (M)$ the space of almost complex structures. This map has the property that the almost complex structure $ret (g,\omega)$ is compatible with $\omega$, and that $ret (g _{J}, \omega ) = J$ for $g _{J} = \omega (\cdot, J \cdot) $. Then $ \{(\omega _{t}, ret (g _{t}, \omega _{t})   \} $ is a compatible  family.
   As $ret  $ is continuous in $C ^{\infty} $-topology, $\delta'$
	 can be chosen such that $ \{ ret _{t} (g _{t}, \omega
	 _{t}    \} $ are $C ^{\infty}$ $\delta$-nearby.
\end{proof}
Returning to the proof of the main lemma. Let 
$\delta' < \delta$ be chosen as in  Lemma  \ref{lemma:Ret}
and let $\{(\omega_{t}, J _{t})\}$ be the corresponding family.
Set $$\widetilde{N} = \overline{\mathcal{M}} ^{k} _{1,1}
(\{J _{t} \},A, \beta ) \cap (U \times [0,1]),   $$ where $U$ is as in
Lemma \ref{lemma:NearbyEnergy}. 

Then $\widetilde{N}$ is an open subset of $\overline{\mathcal{M}} ^{k} _{1,1}
(\{J _{t} \},A, \beta )$.  By Lemma \ref{lemma:NearbyEnergy}, 
$$\widetilde{N} = \overline{\mathcal{M}} ^{k} _{1,1}
(\{J _{t} \},A, \beta ) \cap (\overline{U}  \times [0,1]),   $$ 
so that $\widetilde{N} $ is also closed.  

Finally, $\sup _{(u,t) \in \widetilde{N}} e _{g _{t}}
(u) < \infty  $, by condition \ref{property:ebounded} of $U$, and since
$\{e _{g _{t}} \}$, $t \in [0,1] $   is a continuous family.
Consequently $\widetilde{N} $ is compact by the Gromov
compactness theorem. Resetting $\delta:=\delta'$, we are then
done with the proof of the main lemma. 
\end{proof}

\begin{proposition} \label{thm:nearbyGW} 
Given an almost complex manifold $M,J$  suppose that $N \subset \overline{\mathcal{M}} ^{k} _{1,1}  (J,A)
$ is open and compact. Suppose also that $GW ^{k} _{1,1} (N,
J, A, \beta ) \neq 0$.
Then there is a $\delta>0$ s.t. whenever $J'$ is 
$C ^{2}$ $\delta$-close to $J$, there exists $u \in
\overline {\mathcal{M}} ^{k} _{1,1} (J',A, \beta ) $. 
\end{proposition}
\begin{proof} 
For $N$ as in the hypothesis, let $U$, $\delta$ and $\widetilde{N} $ be as in Lemma 
\ref{lemma:NearbyEnergyDeformation}, then by the conclusion
of that lemma and by Lemma 
\ref{prop:invariance1} $$GW ^{k} _{1,1} (N_1, J', A, \beta ) =  
GW ^{k} _{1,1} (N, J, A, \beta ) \neq 0,$$ where $N _{1} = 
\widetilde{N} \cap \overline{\mathcal{M}} ^{k} _{1,1} 
(J _{1},A, \beta)  $. 
\end{proof}


\section {Elliptic curves in the lcs-fication of a 
contact manifold and the Fuller index} \label{sectionFuller}  
The following elementary result is crucial  for us.
\begin{lemma} \label{lemma:Reeb}
Let $(M,\lambda, \alpha, J)$ be a tamed first kind  
$\lcs$ manifold. Then every 
non-constant (nodal)  
$J$-holomorphic curve $u: \Sigma  \to M$ is a Reeb 2-curve. 
\end{lemma}

\begin{proof} [Proof of Lemma \ref{lemma:Reeb}]
Let $u: \Sigma \to M$ be a non-constant, nodal $J$-curve.  By Lemma \ref{lemma:calibrated} it is enough to show that $[u ^{*}
\alpha] \neq  0$.   Let $\widetilde{M}$ denote the  $\alpha$-covering space of $M$, that is the space of equivalence classes of paths $p$ starting at $x _{0} \in M $, with a pair $p _{1}, p _{2}  $ equivalent if $p _{1} (1) = p _{2} (1)  $ and $$ \int _{[0,1]} p _{1} ^{*} \alpha  =  \int _{[0,1]} p _{2} ^{*} \alpha.$$
Then the lift of $\omega$ to $\widetilde{M} $ is $$\widetilde{\omega}= \frac{1}{f} d (f\lambda), $$
where $f= e ^{-g} $ and where $g$ is a primitive for the lift $\widetilde{\alpha} $ of $\alpha$ to $\widetilde{M} $, that is $\widetilde{\alpha} =dg $.
In particular $\widetilde{\omega} $ is conformally symplectomorphic to an exact symplectic form on $\widetilde{M}$. So if $\widetilde{J}$ denotes the lift of $J$, any closed $\widetilde{J} $-curve is constant by Stokes theorem.
Now if $[u ^{*} \alpha] = 0$ then $u$ has a lift 
to a $\widetilde{J} $-holomorphic map $v: \Sigma \to \widetilde{M}  $.
Since $\Sigma$ is closed, it follows by the above 
that $v $ is constant, so that $u$ is constant, 
which is impossible. 
\end{proof}

\subsection{Preliminaries on Reeb tori}
\label{section_preliminariesReebtori}
Let $(M=C \times S ^{1}, \lambda, \alpha )$ be the
lcs-fication of $(C, \lambda)$. For  $\beta \in \pi _{1} (C)$ we set $A ^{1}
_{\beta} = \beta  \otimes [S	^{1}] \in H _{2} (M,
\mathbb{Z})$.
Let $\mathcal{O} (R ^{\lambda}, \beta)$, be the orbit space
as in Section \ref{sec:Fixed Reeb strings}.
Let $J ^{\lambda }$ on $C \times S ^{1}$ be as in Section
\ref{sec:Reeb holomorphic tori simple}.
We have a map:
\begin{equation} \label{eq:calP}
   \mathcal{P}: \mathcal{O}  (R ^{\lambda}, \beta) \to 
   \overline{\mathcal{M}} ^1 _{1,1}  (J ^{\lambda}, A ^{1}
	 _{\beta}, \beta), \quad \mathcal{P} (o) = u _{o},
\end{equation}
for $u _{o}$ the Reeb torus as previously. We can say more:
\begin{proposition} \label{prop:abstractmomentmap}
For any $(\lambda,\alpha)$-admissible $J$ there is a natural
bijection: \footnote{It is in fact an equivalence of 
the corresponding topological action groupoids, but we do not need this explicitly.}
\begin{equation*}
   \mathcal{P}: \mathcal{O}  (R ^{\lambda}, \beta) \to 
   \overline{\mathcal{M}} ^1 _{1,1}  (J, A ^{1}
	 _{\beta}, \beta), 
\end{equation*}
with $\mathcal{P} $ the map \eqref{eq:calP} in the case $J=J
^{\lambda }$.
 (Note that there is an analogous bijection
$\mathcal{O}(R ^{\lambda}, \beta) \to \overline{\mathcal{M}}
^n _{1,1}  (J, A ^{n} _{\beta}, \beta ),$ for $n>1$, where $A ^{n}
_{\beta} = n \cdot \beta \otimes [S	^{1}])$.
\end{proposition}

In the particular case of $J ^{\lambda} $, we see that all elliptic curves in $C \times S ^{1} $ are Reeb tori, and hence the underlying complex structure on the
domain is ``rectangular''. That is, they are quotients of the 
complex plane by a rectangular lattice. This stops being the
case when we consider generalized Reeb tori in Section
\ref{sec:Mapping tori and Reeb 2-curves} for the mapping
torus of some strict contactomorphism. Moreover, for more general
compatible complex structures we might have nodal degenerations. 

\begin{proof}[Proof of Proposition \ref{prop:abstractmomentmap}] 
We define $\mathcal{P} (o)$ to be the class represented by
the unique up to isomorphism $J$-holomorphic curve $u: T ^{2} \to M$ determined by the conditions:
\begin{itemize}
	\item $u$ is charge 1.
	\item The image of $u$ is the image $\mathcal{T} $ of the map $u _{o}: T ^{2} \to
	M$, $(s, t)  \to  (o  (s), t)$, i.e. the image of the Reeb
	torus of $o$.
	\item The degree of the map $u: T ^{2} \to \mathcal{T}
	$ is the multiplicity of $o$.
\end{itemize}
We need to show that $\mathcal{P} $ is bijective.
Injectivity is automatic.
Suppose we have a curve $u \in 
\overline{\mathcal{M}}_{1,1} ^1  ({J}, A, \beta), $ represented by $u: \Sigma \to 
M $. By Lemma \ref{lemma:Reeb} $u$ is a Reeb 2-curve. 
Then $u$ has 
no spherical components, as such a component corresponds to
a $J ^{\lambda }$-holomorphic map $u': \mathbb{CP} ^{1} \to
M$,  which by Lemma \ref{lemma:Reeb} is also a Reeb
2-curve, and this is impossible by second property in the definition. 

We first show that $u$ is a finite covering map onto the
image of some Reeb torus $u _{o}$.

By Lemma \ref{lemma:ReebCurveRational} 
normalization $\widetilde{u} $ is also a Reeb 
2-curve. If $u$ is not normal then $\widetilde{u}$ is a Reeb
2-curve with domain $\mathbb{CP} ^{1} $, which is impossible
by the argument above. Hence $u$ is normal.

By the charge $1$ condition $pr _{S ^{1} } 
\circ u $ is surjective, where $pr _{S ^{1} }: C 
\times S ^{1} \to S ^{1}   $ is the projection. 
By the Sard theorem we have a regular value $ t_{0} \in S ^{1}   $, so 
that $ u ^{-1} \circ pr _{S ^{1} } ^{-1}  (t  
_{0}) $ contains  an embedded circle  $S _{0} 
\subset \Sigma $. 
Now $d (pr _{S ^{1} } \circ u )$ 
is surjective onto $T _{t _{0}} S ^{1}$ along 
$T \Sigma| _{S _{0} } $. And so by first property of $u$ being
a Reeb 2-curve, $o = pr _{C} 
\circ u| _{S _{0}} $ has non-vanishing 
differential $d(o)$. Moreover, again by the first property, $o$ is tangent to $\ker 
d \lambda $.  It follows that $o$ is an unparametrized 
$\lambda $-Reeb orbit. 

Also, the image of $d (pr _{C} 
\circ u)$ is in $\ker d \lambda $ from which it 
follows that $\image d (pr _{C} \circ u)= 
\image d(o) $. By Sard's theorem and by basic differential
topology it follows that the image of $u$ is contained in the 
image of the Reeb torus $u _{o}$, which is an embedded
2-torus $\mathcal{T} $.

By $J ^{\lambda }$-holomorphicity of $u$, since
$\Sigma \simeq T ^{2}$,  and by basic
complex analysis of holomorphic maps $T ^{2} \to T ^{2}$,
$u$ is a holomorphic covering map onto $\mathcal{T} $, of degree $\deg u$.

Let $\widetilde{o} $ be $\deg u$ cover of $o$.
Then $\mathcal{P} (\widetilde{o})$ is also represented by
a degree $\deg u$, charge one holomorphic covering map $u': T
^{2} \to \mathcal{T} $. By basic covering map theory there is
a homeomorphism of covering spaces:
\begin{equation*}
\begin{tikzcd}
T ^{2} \ar[r, "f"] \ar [d, "u"] & T ^{2} \ar [dl, "u'"]  \\
\mathcal{T}  &     .
\end{tikzcd}
\end{equation*}
Then $f$ is a biholomorphism, so that $u,u'$ are equivalent.
\end{proof}

\begin{proposition} \label{prop:regular} Let $(C, \xi)$ be a 
general contact manifold. If $\lambda$ is a non-degenerate 
contact 1-form for $\xi$ then all the elements of $\overline{\mathcal{M}}_{1,1} ^1   
( J ^{\lambda} , {A}, \beta )$ are regular curves. Moreover, if 
$\lambda$ is degenerate then
for a period $c$ Reeb orbit $o$, the kernel of the 
associated real linear Cauchy-Riemann operator for the Reeb 
torus $u _{o} $ is naturally identified with the 
1-eigenspace of $\phi _{c,*} ^{\lambda}  $ - the time $c$ 
linearized return map $\xi (o (0)) \to \xi (o (0)) $
induced by the $R^{\lambda}$ Reeb flow.
\end{proposition}

\begin{proof}  
We already know by Proposition \ref{prop:abstractmomentmap} that all  
$u \in \overline{\mathcal{M}}_{1,1} ^1   (J ^{\lambda} 
   , {A}, \beta )$ are equivalent to Reeb tori. In 
   particular, such curves have a representation by a $J ^{\lambda} $-holomorphic map $$u: (T ^{2},j) \to (Y = C \times S ^{1}, J ^{\lambda}).  $$
Since each $u$ is immersed we may naturally get a splitting $u ^{*}T (Y) \simeq N \times T (T ^{2})   $,
using the $g _{J} $ metric, where $N \to T ^{2}  $ denotes the pull-back, of the  $g _{J} $-normal bundle to $\image u$, and which is identified with the pullback of the distribution $\xi _{\lambda} $ on $Y$, (which we also call the co-vanishing distribution).

The full associated real linear Cauchy-Riemann operator takes the
form:
\begin{equation} \label{eq:fullD}
   D ^{J}_{u}: \Omega ^{0} (N  \oplus T (T ^{2})  ) \oplus T _{j} M  _{1,1}   \to \Omega ^{0,1}
   (T(T ^{2}), N \oplus T (T ^{2}) ). 
\end{equation}
This is an index 2 Fredholm operator (after standard Sobolev
completions), whose restriction to $\Omega
^{0} (N \oplus T (T ^{2})  )$ preserves the splitting, that is the
restricted operator splits as 
\begin{equation*}
D \oplus D':   \Omega ^{0} (N) \oplus \Omega ^{0} (T (T ^{2})  )    \to \Omega ^{0,1}
(T (T ^{2}), N ) \oplus \Omega ^{0,1}(T (T ^{2}), T (T ^{2}) ).
\end{equation*}
On the other hand the restricted Fredholm index 2 operator 
\begin{equation*}
\Omega ^{0} (T (T ^{2})) \oplus T _{j} M  _{1,1}  \to \Omega ^{0,1}(T (T ^{2}) ),
\end{equation*}
is surjective by classical Teichmuller theory, see also
\cite [Lemma 3.3]{cite_WendlAutomatic} for a precise argument in this setting.
It follows that $D ^{J}_{u}  $ will be surjective
if  
the restricted Fredholm index 0 operator
\begin{equation*}
D: \Omega ^{0} (N) \to \Omega ^{0,1}
(N),
\end{equation*}
has no kernel.

The bundle $N$ is symplectic with symplectic form on
the fibers given by restriction of $u ^{*} d \lambda$, and together with $J
^{\lambda} $ this gives a Hermitian structure $(g _{\lambda}, j _{\lambda} )$ on $N $. We have a
linear symplectic connection $\mathcal{A}$ on $N$, which over the slices $S ^{1}
\times \{t\} \subset T ^{2} $ is induced by the  pullback
by $u$ of the linearized $R  ^{\lambda} $ Reeb flow. Specifically the $\mathcal{A}$-transport map from the fiber $N  _{(s _{0} , t)}  $ to the fiber $N  _{(s _{1}, t)}  $ over the path $ [s _{0}, s _{1} ]
   \times \{t\} \subset T ^{2} $,  is given by $$(u_*| _{N  _{(s _{1}, t)}  }) ^{-1}  \circ (\phi ^{\lambda}
   _{c(s _{1}  - s _{0})})_* 
\circ u_*| _{N _{(s _{0} , t  )}  }, $$ 
where $\phi ^{\lambda} 
   _{c(s _{1}  - s _{0})} $ is the time $c \cdot (s _{1}  - s _{0} )$ map for the $R ^{\lambda} $ Reeb flow, where $c$ is the period of the Reeb orbit $o _{u} $,
   and where $u _{*}: N \to TY $ denotes the natural map, (it is the universal map in the pull-back diagram.)

The connection $\mathcal{A}$ is defined to be trivial in the $\theta
_{2} $ direction, where trivial means that the parallel transport  maps are
the
$id$ maps over $\theta _{2} $ rays.  In particular the curvature $R _{\mathcal{A}} $, understood as a lie algebra valued 2-form, of this connection
vanishes. The connection $\mathcal{A}$ determines a real linear CR operator $D _{\mathcal{A}} $ on
$N$ in the standard way, take the complex anti-linear part of
the vertical differential of a section. Explicitly, 
$$D _{\mathcal{A}}: \Omega ^{0} (N) \to \Omega ^{0,1}
(N), $$
   is defined by $$D _{\mathcal{A}} (\mu) (p) = j _{\lambda} \circ \pi ^{vert} (\mu (p)) \circ d\mu (p) - \pi ^{vert} (\mu (p)) \circ d\mu (p) \circ j, $$ where $$\pi ^{vert} (\mu (p)): T _{\mu (p)} N \to T ^{vert} _{\mu(p)}  N \simeq N $$ is the $\mathcal{A}$-projection, and where $T ^{vert} _{\mu (p)}  N$ is the kernel of the projection $T _{\mu (p)} N \to T _{p}  \Sigma$.
It is elementary to verify that 
the operator $D _{\mathcal{A} }$ is Fredholm $0$ with
the kernel isomorphic to the kernel of $D$. See also \cite [Section
10.1]{cite_SavelyevOh} for a computation of this kind in much greater generality.

We have a differential 2-form $\Omega$ on the total space of $N$ 
defined as follows. On the fibers $T ^{vert} N$, 
$\Omega= u _{*}  \omega $, for $\omega= d_{\alpha} \lambda $, and for $T ^{vert} N \subset TN$ denoting the vertical tangent space, or subspace of vectors $v$ with $\pi _{*} v =0 $, for $\pi: N \to T ^{2} $ the projection. While on the $\mathcal{A}$-horizontal distribution 
$\Omega$ is defined to vanish.
The 2-form $\Omega$ is closed, which we may check explicitly by using that $R _{\mathcal{A}} $ vanishes
to obtain local symplectic trivializations of $N$ in which $\mathcal{A}$ is trivial.
Clearly $\Omega$ must vanish on the
0-section since it is a $\mathcal{A}$-flat section. But any section is homotopic to
the 0-section and so in particular if $\mu \in \ker D$ then $\Omega$
vanishes on $\mu$.  

Since $\mu \in \ker D$, and so its
vertical differential is complex linear, it 
follows that the vertical differential  vanishes. 
To see this note that $\Omega (v, J ^{\lambda}v )
>0$, for $0 \neq v \in T ^{vert}N$ and so if the vertical differential did not vanish we would
have $\int _{\mu} \Omega>0 $. So $\mu$ is
$\mathcal{A}$-flat, in particular the
restriction of $\mu$ over all slices $S ^{1} \times \{t\} $ is
identified with a period $c$ orbit of the linearized at $o$
$R ^{\lambda} $ Reeb flow, and which
does not depend on $t$ as $\mathcal{A}$ is trivial in the $t$ variable. So the kernel of $D$ is identified with the vector
space of period $c$ orbits of the linearized at $o$ $R
^{\lambda} $ Reeb flow, as needed. 
\end{proof}

\begin{proposition} \label{prop:regular2} Let 
$\lambda$ be a contact form on a  $(2n+1)$-fold 
$C$, and $o$ a non-degenerate, period $c$, 
$\lambda$-Reeb orbit, then the orientation of $[u 
_{o} ]$ induced by the determinant line bundle 
orientation of $\overline{\mathcal{M}} ^1 
_{1,1}  ( J ^{\lambda} , {A} ),$ is $(-1) ^{CZ (o) 
-n} $, which is $$\sign \Det (\Id| _{\xi (o(0))}  - \phi _{c, *}
^{\lambda}| _{\xi (o(0))}   ).$$ 
\end{proposition}

\begin{proof}[Proof of Proposition \ref{prop:regular2}]
Abbreviate $u _{o} $ by $u$. Let $N \to T ^{2} $ be 
the vector bundle associated to $u$ as in the proof of Proposition \ref{prop:regular}.
Fix a trivialization $\phi$ of $N$ induced by any trivialization of the
contact distribution $\xi$ along $o$ in the obvious sense: $N$
is the pullback of $\xi$ along the composition $$T ^{2} \to S ^{1}
\xrightarrow{o} C.  $$
Let the symplectic connection $\mathcal{A}$ on $N$ be defined as before. Then the pullback connection $\mathcal{A}' := \phi ^{*} \mathcal{A} $ on $T ^{2} \times \mathbb{R} ^{2n}  $ is a connection whose parallel transport 
paths $p _{t}: [0,1] \to \Symp (\mathbb{R} ^{2n} )$, along the closed loops $S ^{1} \times \{t\} $,
are paths starting at $\id$, and are $t$ independent. And so the parallel transport path of $\mathcal{A}'$ along $\{s\} \times S ^{1} $ is constant, that is $\mathcal{A}' 
$ is trivial in the $t$ variable. We shall call such a
connection $\mathcal{A}'$ on $T ^{2} \times
\mathbb{R} ^{2n}  $ \emph{induced by $p$}.  

   By non-degeneracy assumption on $o$, the map $p(1) $
has no 1-eigenvalues. Let $p'': [0,1] \to \Symp (\mathbb{R} ^{2n} )$ be a path from $p (1)$ to a unitary
map $p'' (1)$, with $p'' (1) $ having no $1$-eigenvalues, and s.t. $p''$
has only simple crossings with the Maslov cycle. Let $p'$ be the concatenation of $p$ and $p''$. We then get  $$CZ (p') - \frac{1}{2}\sign \Gamma (p', 0) \equiv
CZ (p') - n \equiv 0 \mod {2}, $$ since
$p'$ is homotopic relative end points to a unitary geodesic path $h$ starting at
$id$, having regular crossings, and since the number of
negative, positive eigenvalues is even at each regular
crossing of $h$ by unitarity.  Here $\sign \Gamma (p', 0)$
is the index of the crossing form of the path $p'$ at time
$0$, in the notation of
\cite{cite_RobbinSalamonTheMaslovindexforpaths.}.
Consequently,
  \begin{equation} \label{eq:mod2}
  CZ (p'') \equiv CZ (p) -n \mod {2},
  \end{equation} 
    by additivity of
the Conley-Zehnder index. 
   
   Let us then define a free homotopy $\{p _{t} \}$ of $p$ to
$p'$, $p _{t} $ is the concatenation of $p$ with $p''| _{[0,t]} $,
reparametrized to have domain $[0,1]$ at each moment $t$. This
determines a homotopy $\{\mathcal{A}' _{t} \}$ of connections induced by $\{p
_{t} \}$. By the proof of Proposition \ref{prop:regular}, the CR operator $D _{t} $ determined by each  $\mathcal{A}' _{t} $ is surjective except at some finite collection of times $t _{i} \in (0,1) $, $i \in N$ determined by the crossing times of $p''$ with the Maslov cycle, and the dimension of the kernel
of $D _{t _{i} } $ is the 1-eigenspace of $p'' (t _{i} )$, which is 1
by the assumption that the crossings of $p''$ are simple. 
   
 The operator
$D _{1} $ is not complex linear. To fix this we concatenate the homotopy $\{D _{t} \}$ with the homotopy $\{\widetilde{D} _{t}  \}$ defined as follows. Let $\{\widetilde{\mathcal{A}} _{t}  \}$ be a homotopy of $\mathcal{A}' _{1} $ to a
unitary connection $\widetilde{\mathcal{A}} _{1}  $, where the homotopy
$\{\widetilde{\mathcal{A}} _{t}  \}$ is through connections induced by paths
$\{\widetilde{p} _{t} \} $, giving a path homotopy of $p'= \widetilde{p} _{0}  $ to $h$. 
Then $\{\widetilde{D} _{t}  \}$ is defined to be induced by $\{\widetilde{\mathcal{A}} _{t}  \}$.

Let us denote by $\{D' _{t} \}$ the
concatenation of $\{D _{t} \}$ with $\{ \widetilde{D} _{t}  
\}$. By construction, in the second half of the homotopy $\{ {D}' _{t}
\}$, ${D}' _{t}  $ is surjective. And $D' _{1} $ is induced by a unitary connection, since it is induced by unitary path $\widetilde{p}_{1}  $. Consequently, $D' _{1} $ is complex linear. By the above construction,
for the homotopy $\{D' _{t} \}$, $D' _{t} $ is surjective except for
$N$ times in $(0,1)$, where the kernel has dimension one. 
In
particular the sign of $[u]$ by the definition via the determinant
line bundle is exactly $$-1^{N}= -1^{CZ (p) -n},$$
by \eqref{eq:mod2}, which was what to be proved.
\end {proof}

\begin{theorem} \label{thm:GWFullerMain} 
\begin{equation*} 
GW _{1,1} ^{1} (N,J ^{\lambda}, A _{\beta}, \beta ) ([\overline {M} _{1, 1}] 
\otimes [C \times S ^{1} ]) = i (\mathcal{P} ^{-1}({N}), R ^{\lambda}, \beta), 
\end{equation*} where $N \subset \overline{\mathcal{M}} ^1 _{1,1} ( {J} ^{\lambda},
A _{\beta}, \beta  )$ is an open compact set (where $\mathcal{P}$ is as in Proposition \ref{prop:abstractmomentmap}), $i (\mathcal{P} ^{-1}({N}), R ^{\lambda}, \beta)$ is the Fuller index as described in the Appendix below, and where the left-hand side of the equation is the functional as in \eqref{eq:functionals2}.
\end{theorem}
\begin{proof} 
Suppose that ${N} \subset \overline{\mathcal{M}} ^1   _{1,1}
(J ^{\lambda}, A _{\beta}, \beta ) $ is open-compact and consists
of isolated regular Reeb tori $\{u _{i} \}$, corresponding
to orbits $\{o _{i} \}$. Denote by $mult (o _{i}) $ the
multiplicity of the orbits as in Appendix \ref{appendix:Fuller}. Then  we have:
\begin{equation*}
   GW _{1,1} ^{1} (N,   J ^{\lambda}, A _{\beta}, \beta ) ([\overline{M} _{1,1}   ] \otimes [C \times S ^{1} ]) = \sum _{i} \frac{(-1) ^{CZ (o _{i} ) - n} }{mult (o _{i} )},
\end{equation*}
where $n$ half the dimension of $M$, the numerator is as in
\eqref{eq:conleyzenhnder}, and $mult (o _{i} )$ is the order of the corresponding isotropy group,  see Appendix \ref{sec:GromovWittenprelims}.

The expression on the right is exactly the Fuller index $i 
(\mathcal{P}^{-1} (N), R ^{\lambda},  \beta)$.
Thus, the theorem follows for $N$ as above. 
However, in general if $N$ is open and compact then perturbing slightly we obtain a smooth family $\{R ^{\lambda _{t} } \}$, $\lambda _{0} =\lambda $, s.t.
 $\lambda _{1} $ is non-degenerate, that is has non-degenerate orbits. 
And such that there is an open-compact subset $\widetilde{N} $ of 
$\overline{\mathcal{M}}  _{1,1} ^1  (\{J ^{\lambda _{t} } \}, A 
_{\beta}, \beta)$ with $(\widetilde{N} \cap \overline{\mathcal{M}}  _{1,1} 
^1  (J ^{\lambda}, A _{\beta}, \beta ) = N $, see Lemma \ref{lemma:NearbyEnergyDeformation}.
 Then by Lemma \ref{prop:invariance1} if $$N_1=(\widetilde{N} \cap 
 \overline{\mathcal{M}} ^1   _{1,1} (J ^{\lambda _{1} },
 A _{\beta}, \beta ))
$$ we get $$GW ^{1} _{1,1} (N,   J ^{\lambda}, A _{\beta}, \beta
) ([\overline{M}  _{1,1} ] \otimes [C \times S ^{1} ]) = GW
^{1} _{1,1} (N _{1},   J ^{\lambda _{1} }, A _{\beta}, \beta  ) ([\overline{M} _{1,1}  ] \otimes [C \times S ^{1} ]).
$$
By the previous discussion 
\begin{equation*}
   GW ^{1} _{1,1} (N _{1},   J ^{\lambda _{1} }, A _{\beta}, \beta ) ([\overline{M} _{1,1}   ] \otimes [C \times S ^{1} ]) = i (N_1, R ^{\lambda_1},  \beta), 
\end{equation*}
but by the invariance of Fuller index (see Appendix \ref{appendix:Fuller}), $$i (N_1, R ^{\lambda_1},  \beta) = i (N, R ^{\lambda},  \beta).
$$ 
\end{proof}
What about higher genus invariants of $C \times S ^{1} $? Following the proof of Proposition \ref{prop:abstractmomentmap}, it is not hard to see that all $J ^{\lambda} $-holomorphic curves must be branched covers of Reeb tori. If one can show that these branched covers are 
regular when the underlying tori are regular, the
calculation of invariants would be fairly  automatic from
this data. See \cite{cite_WendlSuperRigid}, \cite{cite_WendlChris} where these kinds of regularity calculation are made.



\section{Proofs of main theorems} 
\label{sec:Proofs of theorems on Conformal symplectic Weinstein conjecture}

To set notation and terminology we review the basic definition 
of a nodal curve.
\begin{definition}\label{def:normalization} A 
\textbf{\emph{nodal Riemann surface}}  (without boundary) is a pair 
$\Sigma= (\widetilde{\Sigma }, \mathcal{N} ) $ 
where $\widetilde{\Sigma } $  is a Riemann 
surface, and $\mathcal{N}$ a set of pairs of 
points of $\widetilde{\Sigma }: $ $\mathcal{N} =\{(z _{0} ^{0}, z _{0} ^{1}),  \ldots, (z _{n} ^{0}, z _{n} 
^{1}) \}$, $n _{i} ^{j} \neq n 
_{k} ^{l}$ for $i \neq k$ and all $j,l$. By slight 
abuse, we may also denote by $\Sigma $ the quotient space
$ \widetilde{\Sigma}/\sim $, where the  
equivalence relation is generated by $n _{i} ^{0} 
\sim n _{i} ^{1}$.  
Let $q _{\Sigma}: \widetilde{\Sigma} \to 
(\widetilde{\Sigma}/\sim)$ denote the quotient map.  
The elements $q _{\Sigma } (\{z _{i} 
^{0}, z _{i} ^{1} \}) \in \widetilde{\Sigma}/\sim $,  are called \textbf{\emph{nodes}}.
Let $M$ be a smooth manifold. By a map $u: \Sigma \to M$ of
a nodal Riemann surface $\Sigma$, we mean a set map $u: 
(\widetilde{\Sigma}/\sim) \to M$.  $u$  is called 
smooth or immersion or $J$-holomorphic (when $M$ 
is almost complex) if the map $\widetilde{u}= u \circ q _{\Sigma }$ 
is smooth or respectively immersion or respectively $J$-holomorphic. We call $\widetilde{u} $ 
\textbf{\emph{normalization of $u$}}.  $u$ is called an embedding if $u$ is a topological embedding and its normalization is an immersion.
The cohomology groups of $\Sigma $ are defined as 
$H ^{\bullet} (\Sigma) := H ^{\bullet } (\widetilde{\Sigma}/ \sim)$, likewise with homology. The genus of 
$\Sigma $ is the topological genus of 
$\widetilde{\Sigma}/\sim $.
\end{definition}
We shall say that $(\widetilde{\Sigma }, 
\mathcal{N} ) $ is normal if $\mathcal{N} = 
\emptyset$. Similarly, $u: \Sigma \to 
M$, $\Sigma = (\widetilde{\Sigma }, \mathcal{N} ) 
$  is called \textbf{\emph{normal}} if $\mathcal{N} = \emptyset 
$. The normalization of $u$ is the map of the 
nodal Riemann surface $\widetilde{u}: 
\widetilde{\Sigma} \to M$,  $\widetilde{\Sigma} = 
(\widetilde{\Sigma}, \emptyset )  $. Note that if 
$u$ is a Reeb 2-curve, its normalization $\widetilde{u} $  may not be a Reeb 2-curve (the second condition may fail).

%
\subsection{Mapping tori and Reeb 2-curves} \label{sec:Mapping tori and Reeb 2-curves}
Let $(C, \lambda ) $ be a contact manifold and $\phi$
a strict contactomorphism and let $M = (M _{\phi, 1}, \lambda _{\phi}, 
\alpha )$ denote the 
mapping torus of $\phi$, as also appearing in Theorem
\ref{thm:firstkindtorus}.   More specifically, $M = C \times 
\mathbb{R} ^{} / \sim $, where the equivalence $\sim$ is generated by 
$(x, \theta ) \sim (\phi (x), \theta +1)  $, for more details on the
corresponding lcs structure see for instance
~\cite{cite_BazzoniFirstKind}.
Then $(M, \lambda _{\phi}, \alpha)$ is an integral first kind lcs manifold.

In this case  $\mathcal{V} _{\lambda} = \mathcal{D}$, and in
mapping torus coordinates at a point $(x, \theta ) $,  it
is spanned by $X _{\lambda} = (0, \frac{\partial }{\partial
\theta }), X _{\alpha} = (R ^{\lambda _{\theta }}, 0) $ for $R 
^{\lambda _{\theta}}$ the $\lambda _{\theta}$-Reeb vector field, where
$\lambda _{\theta} = \lambda _{C _{\theta}}$ the fiber over
$\theta$ of the
projection $M \to S ^{1}$. Analogously to the Example
\ref{section:lcsfication} there is an $S ^{1}$-invariant
almost complex structure on $M$, which we call $J ^{\lambda
_{\phi }}$.

We now show that all Reeb 2-curves in $M$ must be of a 
certain type. Let $o: S ^{1} \to C$ be a $\lambda$-Reeb
and suppose that $\image 
\phi ^{n} (o) = \image o$,  for some $n>0$, so that
$$\forall t \in [0,1]: \phi ^{n} (o) (t) = o (t+ \theta
_{0})$$ for some
uniquely determined $\theta _{0} \in [0,1)$.
Let $\widetilde{o}: S ^{1} 
\times [0,n] \to C \times \mathbb{R}  $ be the map
$$\widetilde{o} (t, \tau) = (o (t+ \theta _{0} \cdot
\frac{\tau }{n}), \tau).$$
Then $\widetilde{o} $ is well defined on the quotient
$T ^{2} \simeq S ^{1} \times ([0,n]/0 \sim n) $, and we
denote the quotient map by $u ^{n} _{o}$, called the
\emph{charge $n$ generalized Reeb torus of $o$}. If the
class $[o] \in \pi _{1}  ({C}) = \beta $ we denote by $A ^{n} _{\beta }$ the
class of $u ^{n} _{o}$ in $H _{2} (M, \mathbb{Z} )$. The class $A ^{n} _{\beta }$ can be
defined more generally whenever ${\phi} ^{n} _{*}
({\beta}) = {\beta } $, it is the class of a torus map $T
^{2} \to M$ defined analogously to the map
$u ^{n} _{o}$, but no longer having the Reeb 2-curve
property. We may abbreviate $A ^{1} _{\beta }$ by $A _{\beta
}$.

By construction,  $u ^{n} _{o}$ is a charge $n$ Reeb 2-curve and its image is an embedded $J ^{\lambda
_{\phi}}$-holomorphic torus $\mathcal{T} $. Moreover, $u ^{n} _{o}$ is 
$J ^{\lambda _{\phi}}$-holomorphic with respect to
a uniquely determined complex structure on $T ^{2}$,
similarly to the case of Reeb tori of Section \ref{sec:Reeb holomorphic tori simple}.
However, unlike the case of Reeb tori, this complex
structure is not ``rectangular'' unless $\phi ^{n} \circ o = o$.
\begin{proposition} \label{prop:Topembedding} Let $M = (M
_{\phi}, \lambda _{\phi}, \alpha)$ be the
mapping torus of a strict contactomorphism $\phi $ as above. Then:
\begin{enumerate}
\item Every charge $n$ Reeb 2-curve $u$ in $M$  has
a factorization:
\begin{equation} \label{eq:factorization}
u = u ^{n} _{o} \circ \rho,
\end{equation}
for $\rho: \Sigma \to T ^{2}$ some degree one map, and for
some orbit string $o$ uniquely determined by $u$. 
\item Every element $u \in \overline{\mathcal{M}} ^{n}
_{1,1} (J ^{\lambda _{\phi}}, A ^{n} _{\beta}, \beta)$ is represented (as an
equivalence class in this moduli space) by $u ^{n} _{o}$,
where the latter is as above, for some $o$ uniquely determined.
\item The Fredholm index of the corresponding real linear CR
operator is 2, so that the expected dimension
of $\overline{\mathcal{M}} ^{n}
_{1,1} (J ^{\lambda _{\phi}}, A ^{n} _{\beta}, \beta)$ is 0.
\label{label_fredholm}
\item Let $J$ be a $(\lambda _{\phi}, \alpha)$-admissible
almost complex structure on $M$. There is a natural proper topological embedding:
\begin{equation*}
emb: \overline{\mathcal{M}} ^{n} _{1,1} (J, 
A ^{n} _{\beta}, \beta) \to \mathcal{O} (R
^{\lambda }, \beta ),
\end{equation*}
defined by $u \mapsto o$, where $o$ is uniquely
determined by the condition \eqref{eq:factorization}. \label{part4_propTopEmbedding}
\end{enumerate}
\end{proposition}
\begin{proof}
The proof of part one is completely analogous to the proof
of Proposition \ref{prop:abstractmomentmap}. To prove the second
part, first note that by the first part $u$ has image
$\mathcal{T} = \image u ^{n} _{o}$ for some $n, o$, and
$\mathcal{T} $ is an embedded $J ^{\lambda
_{\phi}}$-holomorphic torus. By basic theory of mappings of
complex tori $u$ must be a covering map $\Sigma \to
\mathcal{T} $. Since we know the charge $n$, as in final
part of the proof
of Proposition \ref{prop:abstractmomentmap}, we may conclude
that $u \simeq u ^{n} _{o}$ for some uniquely determined
$o$, where $\simeq $ is an isomorphism.

We prove Part \ref{label_fredholm}. Note  that $c _{1} (A
^{n} _{\beta }) = 0$, as by construction the complex tangent
bundle along $u ^{n} _{o}$ admits a flat connection, induced
by the natural $\mathcal{G} $-connection on $M _{\phi} \to
S ^{1}$, for $\mathcal{G} $ the group of strict
contactomorphisms of $(C, \lambda )$, 
cf. Proof of Proposition \ref{prop:regular}.
The needed fact then follows by the index/Riemann-Roch theorem.

The last part of the proposition readily follows from the
first part.
\end{proof}
\begin{proof} [Proof of Theorem \ref{theorem_counterexample}]
Let $u: \Sigma \to (M = M _{\widetilde{\phi},1})$ be a Reeb 2-curve in the
mapping torus as in the statement. By part one of
Proposition \ref{prop:Topembedding}, there must be
a generalized charge $n$ Reeb torus in $M$. By definitions
this means that $\widetilde{\phi}  $ has a charge $n$ fixed
Reeb string, so that $\phi $ has a charge $n$ fixed geodesic
string, which is impossible by assumptions.
\end{proof}

\begin{proposition} \label{thm:holomorphicSeifert} 
Let $(C,\lambda) $ be a contact manifold 
with $\lambda$ satisfying one of the following conditions:
\begin{enumerate}
	\item There is a non-degenerate $\lambda $-Reeb orbit. 
	\item $i (N,R ^{\lambda}, \beta) \label{condition:prop1} 
\neq 0$ for some open compact $N \subset \mathcal{O} (R
^{\lambda }, \beta ) $, and some $ \beta$.
\end{enumerate}
Then:
\begin{enumerate}
   \item  Let $(\lambda, \alpha) $ be the lcs-fication of
	 $(C, \lambda ) $. There exists an $\epsilon>0$ s.t. for any tamed exact lcs structure $(\lambda', \alpha', J)$  on $M=C \times 
S ^{1}$, with $(d _{\alpha'} \lambda', J)$ 
$\epsilon$-close to $(d_{\alpha}\lambda, J 
^{\lambda})$ (as in Definition 
\ref{def:deltaclose}), there exists an elliptic, 
$J$-holomorphic $\alpha $-charge 1 curve $u$ in $M$.
\item In addition, if $(M, \lambda', \alpha')$ is first 
kind and has dimension $4$ then $u$ may be 
assumed to be normal and embedded.
\end{enumerate}
\end{proposition}

\begin{proof} 
If we have a closed non-degenerate $\lambda$-Reeb orbit $o$ then we also have an open compact subset 
$N = \{o\} \subset S _{\lambda }$.
Thus suppose that the condition \ref{condition:prop1} 
holds. 

Set $$ 
(\widetilde{N} := \mathcal{P} (N)) \subset \overline{\mathcal{M}} ^{1}  
_{1,1} ( J ^{\lambda}, A _{\beta }, {\beta}),  $$ which is 
an open compact set.
By Theorem \ref{thm:GWFullerMain}, and by the 
assumption that $i (N, R _{\lambda }, \beta ) \neq 0 $ $$GW
^{1} _{1,1} (N,  J ^{\lambda}, A _{\beta }, \beta)  \neq 0. $$   The
first part of the conclusion then follows by Proposition \ref{thm:nearbyGW}. 

We now verify the second part.  Suppose that $M$ 
has dimension 4. Let $U$ be an 
$\epsilon$-neighborhood of $(\lambda, \alpha, J 
^{\lambda } ) $, for $\epsilon$  as given in the first part, and let $(\lambda', \alpha', J) \in U $. Suppose that $u 
\in \overline{\mathcal{M}} ^{1}  _{1,1} ( J, \beta ) 
$. Let $\underline{u}$ be a simple 
$J$-holomorphic curve covered by $u$, (see for 
instance ~\cite[Section 2.5]
{cite_McDuffSalamonJholomorphiccurvesandsymplectictopology}.

For convenience, we now recall the adjunction inequality. 
\begin{theorem} [McDuff-Micallef-White
\cite{cite_MicallefWhite}, \cite{cite_McDuffPositivity}] 
 Let $(M, J)$ be an almost complex 4-manifold and let
 $A \in H_2(M)$ be a homology class that is 
 represented by a simple J-holomorphic curve $u: 
 \Sigma \to M$.  Let $\delta (u) $ denote the number of 
 self-intersections of $u$, then \begin{equation*}
2\delta (u) - \chi (\Sigma) \leq A\cdot A -c _{1} (A),
\end{equation*}
with equality if and only if $u$  is an immersion with only transverse self-intersections. 
\end{theorem}
In our case $A= A _{\beta }$ so that $c _{1} (A) =0$ and $A \cdot A = 0$. If ${u}$ is not normal its normalization is of the form 
$\widetilde{{u}}: \mathbb{CP} ^{1} \to M$ with 
at least one self intersection and with $0 = 
[\widetilde{{u}} ] \in H _{2} (M)$, but this contradicts positivity of 
intersections. So $u$ and hence $\underline{u}$ are normal.
Moreover, the domain $\Sigma'$ of $\underline {u}
$ satisfies: $\chi 
(\Sigma')= \chi (T ^{2}) = 0$, so that $\delta 
(\underline{u})=0$, and the above inequality is an 
equality. In particular $\underline{u}$ is an embedding,
which of course implies our claim.

\end{proof}

\begin{proof} [Proof of Theorem \ref{thm:C0Weinstein}]
Let $$U \ni (\omega _{0}:= d_{\alpha} {\lambda}, J _{0}:= J ^{\lambda})$$ be a set of pairs 
$(\omega, J)$ satisfying the 
following:
\begin{itemize}
\item $\omega$ is a first kind $\lcs$ structure.  
\item For each $(\omega, J) 
\in U$, $J$ is $\omega$-compatible and admissible.
\item Let $\epsilon$ be chosen as in the first 
part of Proposition \ref{thm:holomorphicSeifert}. Then each 
$(\omega, J) \in U$ is $\epsilon$-close to 
$(\omega _{0}, J _{0})$, (as in Definition 
\ref{def:deltaclose}). 
\end{itemize}
To prove the theorem we need to construct a map 
$E: V \to \mathcal{J} (M)$, where $V$ is some  
neighborhood of $\omega _{0}$  in the space 
$(\mathcal{F} (M), d _{\infty })$ (see Definition 
\ref{def:LM}) and where $$\forall \omega \in V: (\omega, E (\omega)) \in 
{U}.$$ 
As then Proposition \ref{thm:holomorphicSeifert} tells us that 
for each $\omega \in V$, there is a class $A$, $E 
(\omega) $-holomorphic, elliptic curve $u$ in $M$.  
Using Lemma \ref{lemma:Reeb} we would then 
conclude that there is an elliptic Reeb 2-curve $u$ in $(M, 
\omega)$. If $M$ has dimension 4 then in addition $u$ may be 
assumed to be normal and embedded.
If $\omega$ is integral, by Proposition 
\ref{thm:holomorphicSeifert}, $u$ may be assumed to be 
charge 1. And so we will be done.

Define a metric $\rho _{0} $ measuring the distance between subspaces  $W _{1}, W _{2}  $, of same dimension, of an inner product space $(T,g)$ as follows. $$\rho _{0} (W _{1}, W _{2}  ) := |P _{W _{1} } - P _{W _{2} }  |, $$ for $|\cdot|$ the $g$-operator norm, and $P _{W _{i} } $ $g$-projection operators onto $W _{i} $. 

Let $\delta>0$ be given. Suppose that $\omega=d 
^{\alpha'} \lambda' $ is a first kind $\lcs$ 
structure $\delta$-close to $\omega _{0} $ for the  metric
$d _{\infty} $. Then $\mathcal{V} _{\lambda'}, \xi 
_{\lambda'}$ are smooth distributions by the 
assumption that $(\alpha', \lambda') $ is a $\lcs$ 
structure of the first kind and $TM = \mathcal{V} 
_{\lambda'} \oplus \xi 
_{\lambda'}$.  Moreover, $$\rho
_{\infty} (\mathcal{V} _{\lambda'}, \mathcal{V}
_{\lambda}) < \epsilon _{\delta}  $$ and  $$\rho
_{\infty} (\xi _{\lambda'}, \xi _{\lambda})
< \epsilon _{\delta}  $$  where $\epsilon _{\delta} \to
0 $ as $\delta \to 0$, and where $\rho _{\infty} $ is the
$C ^{\infty }$ analogue of the metric $\rho _{0}$, for the
family of subspaces of the family of inner product spaces $(T _{p} M,g) $.

Then choosing $\delta$ to be suitably be small, for each $p \in M$ we have an isomorphism $$\phi (p): T _{p} M \to T _{p} M,   $$ $\phi _{p}:= P _{1} \oplus P _{2}   $, for $P _{1}: \mathcal{V}_ {\lambda _{0} } (p) \to \mathcal{V} _{\lambda'} (p) $, $P _{2}: \xi _{\lambda _{0}} (p) \to \xi _{\lambda'} (p)   $ the $g$-projection operators. 
Define $E (\omega) (p):= \phi (p)_*J _{0}  $.
Then clearly, if $\delta$ was chosen to be 
sufficiently small, if we take $V$ to be the $\delta $-ball in
$(\mathcal{F} (M), d _{\infty }) $ centered at $\omega
_{0}$, then it has the needed property.

%
%
\end{proof}

\begin{definition}\label{def:classifyingmap}
Let $\alpha $ be a scale integral closed 1-form on a closed
smooth manifold $M$. Let $0 \neq c \in \mathbb{\mathbb{R} ^{}} $ be such that $c \alpha $ is integral. A \textbf{\emph{classifying map}}  $p: M \to 
S ^{1}$ of $\alpha$ is a smooth map s.t.  $c \alpha = p ^{*}
d\theta $. A map $p$ with these properties is of course not
unique. 
\end{definition}
\begin{lemma}
   \label{lemma:ReebCurveRational} Let $u: \Sigma 
   \to M$ be a Reeb 2-curve in a closed, scale integral,
   first kind lcs manifold $(M, \lambda, \alpha ) 
   $, then its normalization  $\widetilde{u}: 
   \widetilde{\Sigma}  \to M$ is a Reeb 2-curve. 
\end{lemma}

\begin{proof}
	 By Lemma \ref{lemma:alphaclass} we have a surjective 
   classifying map $p: M \to S ^{1}$ of $\alpha$. 
   Note that the fibers $M _{t}$ of $p$, for all $t \in S ^{1}$,  
   are contact with contact form $\lambda _{t} = \lambda| _{C _{t}} 
   $, as $0 \neq \omega ^{n} = \alpha  \wedge 
   \lambda  \wedge d \lambda ^{n-1} $ and $c \cdot \alpha 
   = 0$ on $M _{t}$, where $c$ is as in the definition of
	 $p$.

   Let $\widetilde{u}: \widetilde{\Sigma }  \to M$ 
   be the normalization of $u$. Suppose it is not 
   a Reeb 2-curve, which in this case, by definitions, just means 
   that $0 =  [\widetilde{u} ^{*} \alpha] \in H 
   ^{1} (\widetilde{\Sigma }, \mathbb{R} ^{} ) $.
   Since $0 \neq  [{u} ^{*} \alpha] \in H ^{1} 
   ({\Sigma }, \mathbb{R} ^{} )$, some node $z 
   _{0}$ of $\Sigma $ lies on closed loop $o: S 
   ^{1} \to \Sigma $ with $\langle [o], [u 
   ^{*}\alpha ]   \rangle \neq 0$. 
	 
	 Let $q _{\Sigma }: \widetilde{\Sigma } \to \Sigma $ be
	 the quotient map as previously appearing.
	 In this case, we may find a smooth embedding $\eta: D ^{2}  
   \to \widetilde{\Sigma}$, s.t. $q _{\Sigma } \circ \eta
	 (D ^{2})| _{\partial D ^{2}} $ is a component of
	 a regular fiber $C _{t}$, of the classifying map $p':
	 {\Sigma}  \to S ^{1}$ of ${u}  ^{*} \alpha $. See Figure \ref{figure:disk},
 $\eta (D
	 ^{2}) $  is a certain disk in $\widetilde{\Sigma} $,
	 whose interior contains  an element of $\phi ^{-1}(z
	 _{0})$. 	 
	\begin{figure}[h]
 \includegraphics[width=3.0in]{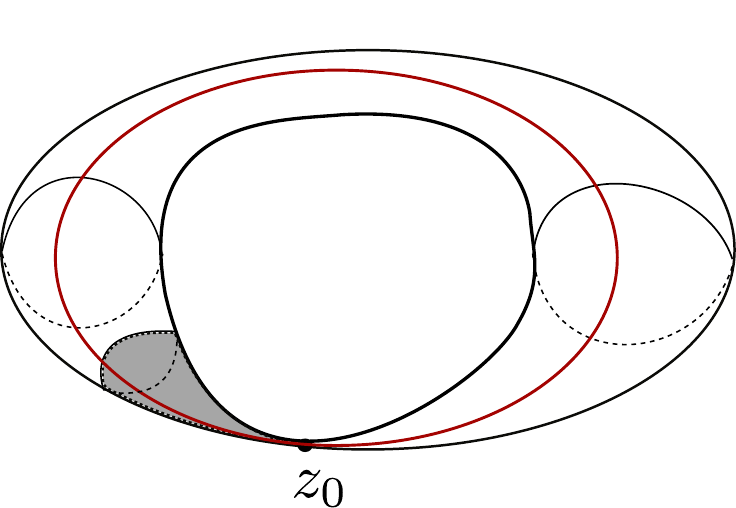}
 \caption {The figure for $\Sigma $. The gray shaded area is
 the image $q _{\Sigma} \circ \eta (D ^{2}) $. The red shaded curve is the image of the closed loop $o$ as above.} \label{figure:disk}
\end{figure}   
   

Then analogously to the proof of Proposition \ref{prop:abstractmomentmap} $\widetilde{u}  \circ \eta| _{\partial
D ^{2}} $ is a (unparametrized)  $\lambda _{t}$-Reeb orbit in 
$M _{t}$. (The classifying maps can be arranged, such that $u(C _{t}) \subset M _{t}$.) And in particular $\int _{\partial D} 
\widetilde{u} ^{*} \lambda \neq 0$. Now 
$u$ (not $\widetilde{u} $)  is a Reeb 2-curve, and the first
condition of this implies 
that $\int _{D}  d\widetilde{u} ^{*} \lambda =0 
$, since $\ker d \lambda $ on $M$ is spanned by 
$X _{\lambda }, X _{\alpha }$. So we have a 
contradiction to Stokes theorem. Thus,  
$\widetilde{u} $ must be a Reeb 2-curve.
   
   %
\end{proof}

%


\begin{proof} [Proof of Theorem \ref{thm:basic0}] 
Let $(C, \lambda) $ and $\phi$ be as in the 
hypothesis. 
Let $\{\phi _{t}\}$, $t \in [0,1] $ be a smooth
family of strict contactomorphisms $\phi _{0} = id$, $\phi
_{1} = \phi$. This gives a smooth
fibration $\widetilde{M} \to [0,1] $, with fiber over $t
\in [0,1] $: $M _{\phi _{t},1} $, which is moreover endowed
with the first kind lcs structure $(\lambda _{\phi _{t}}, 
\alpha )$, where this is the ``mapping torus structure'' as above.  Let $tr: \widetilde{M} \to (C \times
S ^{1}) \times [0,1]  $ be a smooth trivialization,
restricting to the identity $C \times S ^{1} \to C \times
S ^{1}$  over $0$. Pushing forward by the bundle map $tr$,
the above mentioned family of lcs structures, we get
a smooth family $\{(\lambda _{t}, \alpha) \} $, $t \in [0,1]
$,   of first kind integral lcs structures on $C \times
S ^{1}$, with $(\lambda _{0}, \alpha) = (\lambda, \alpha
)  $ the standard lcs-fication of $\lambda $.


Fix a family $\{J ^{\lambda _{t}} \}$ of almost complex
structures on $C \times S ^{1}$ with each $J ^{\lambda
_{t}}$ admissible with respect $(\lambda _{t}, \alpha) $. 
Let $N \subset \mathcal{O} (R ^{\lambda }, \beta ) $ be an open
compact set satisfying $i (N, R ^{\lambda }, \beta ) \neq
0$.
The embedding $emb$ from part \ref{part4_propTopEmbedding} of
Proposition \ref{prop:Topembedding} induces
a proper embedding $$\widetilde{emb}: \widetilde{\mathcal{M}
} = \overline{ \mathcal{M}} ^{1} _{1,1} (\{J ^{\lambda _{t}}\},
A _{\beta}, \beta ) \to \mathcal{O} (R
^{\lambda}, \beta ) \times [0,1],  $$ 
defined by $\widetilde{emb} (u,t) = (emb (u),t)$.

Set $N _{0} = \mathcal{P} (N)  \subset \overline{
\mathcal{M}} ^{1} _{1,1} (J ^{\lambda _{0}}, A _{\beta}, \beta )
\subset \widetilde{\mathcal{M} } $. So that by construction
$\widetilde{emb} (N _{0}) = N 
\times \{0\}$. Set $$\widetilde{N} = \widetilde{emb}
^{-1}(\widetilde{emb}(\widetilde{\mathcal{M}
}) \cap (N \times [0,1])),  $$ then this is an open and
compact subset of $\widetilde{\mathcal{M} } $.  
And by construction $\widetilde{N} \cap
\overline{\mathcal{M}} ^{1} _{1,1} (J ^{\lambda
_{0}}, A _{\beta}, \beta) = N	_{0}$. Set $N _{1} = \widetilde{N}
\cap  \overline{\mathcal{M}} ^{1} _{1,1} (J ^{\lambda
_{1}}, A _{\beta}, \beta ) $. 


It follows by Theorem \ref{thm:GWFullerMain} that 
$$GW _{1,1} ^{1} (N _{0}, J ^{\lambda _{0}}, A _{\beta}, \beta )
([\overline {M} _{1, 1}] \otimes [C \times S ^{1} ])
= i (N, R ^{\lambda }, \beta )  \neq 0.$$
Then applying Lemma \ref{prop:invariance1} we get that 
\begin{equation*}
GW _{1,1} ^{1} (N _{1}, J ^{\lambda _{1}}, A _{\beta}, \beta)
([\overline {M} _{1, 1}] \otimes [C \times S ^{1} ])
 \neq 0.
\end{equation*}
By part two of Proposition \ref{prop:Topembedding} there is a charge 1 generalized Reeb torus $u _{o}$ in $M$.
In particular, $\image \phi (o) = \image (o)$. Also 
by construction, $o \in N$ and so we are done.
\end{proof}

\begin{proof} [Proof of Theorem \ref{thm:infinitetype}] 
If $R ^{\lambda }$ is finite type and $i (R ^{\lambda },
\beta ) \neq 0$ then the theorem follows
immediately by Theorem \ref{thm:basic0}.

We leave the full definition of infinite type vector fields
to the reference ~\cite [Definition 2.5]{cite_SavelyevFuller}. 
We have that $R ^{\lambda}$ is infinite type in class $\beta
$. WLOG assume that it is positive
infinite type. In particular, we may find
a perturbation $X ^{a}$ of $R ^{\lambda }$ together with
a homotopy $X _{t}$, $t \in [0,1] $, s.t.:
\begin{enumerate}
	\item $X _{0} = X ^{a}$, $X _{1} = R ^{\lambda }$. 
	\item $\mathcal{O}  (X ^{a},a, {\beta}) = \{o \in
	\mathcal{O}  (X ^{a}, {\beta}) \,|\,  A (o) \leq a\} $ is
	discrete, where $A$ is the period map as in the
	introduction. 
	\item Each $o \in
	\mathcal{O}  (X ^{a},a, {\beta})$ is contained in a 
	non-branching open compact subset $K _{o} \subset
	\mathcal{O} (\{X ^{a} _{t}\}, \beta) $. Where the latter
	means that: 
	 \begin{enumerate}
	\item ${K} _{o}  \cap  \mathcal{O}  (X _{1}, \beta) $ is connected.
	\item For $o, o' \in \mathcal{O}  (X ^{a}, a, \beta) $ $K _{o} = K _{o'}$
	or $K _{o} \cap K _{o'} = \emptyset$.
\item $\mathcal{O}  (X _{1}, \beta) = \cup _{o}  ({K} _{o} \cap
\mathcal{O}(X _{1}, \beta)).$
\item $(K _{o} \cap \mathcal{O} (X ^{a}, \beta )) \subset
\mathcal{O} (X ^{a},
a, \beta ) $.
\end{enumerate}

	\label{condition:non-branching}
	\item $\sum _{o \in \mathcal{O}  (X ^{a},a, {\beta})}  i (o) > 0$.
\end{enumerate}
Set $N := \mathcal{O}  (X ^{a},a,\beta)$, then by the
condition \ref{condition:non-branching} and by the
non-branching property, $$N = \sqcup _{i \in  \{1,
\ldots,n\}} \mathcal{O}  (X ^{a},\beta) \cap  K _{o
_{i}},$$ disjoint union for some $o _{1},
\ldots o _{n} \in \mathcal{O}  (X ^{a},a,\beta)$.  Set $\widetilde{N} := \cup _{i \in  \{1,
\ldots,n\}} K _{o _{i}}$. Then
$\widetilde{N} \cap \mathcal{O}  (X ^{a}, \beta  ) = N $.
Set $N _{1} := \widetilde{N} \cap \mathcal{O}  (R
^{\lambda}, \beta)  $, then this is an open compact subset
of $\mathcal{O} (X _{1}, \beta ) $. 

Finally, using invariance of the Fuller index we get that $i
(N _{1}, R ^{\lambda},
{\beta }) \neq 0.$ Then the result follows by Theorem \ref{thm:basic0}. 
\end{proof}

\begin{proof} [Proof of Theorem \ref{thm:basis-1}]  Suppose
that $\lambda$ is Morse-Bott and we have an open
compact component $N \subset \mathcal{O} (R ^{\lambda}, \beta)  $,
with $\chi (N) \neq 0 $ so that $i (N, R ^{\lambda}, \beta) \neq
0$, ~\cite [Section 2.1.1] {cite_SavelyevFuller}. If
$\lambda' $ is sufficiently $C ^{1}$ nearby to $\lambda$ then we may find open compact $N' \in \mathcal{O} (R
^{\lambda'}) $ s.t. $$i (N', R ^{\lambda'}, \beta) = i (N,
R ^{\lambda }, \beta ) \neq 0. $$  See ~\cite[Lemma 1.6]
{cite_SavelyevFuller}. Then the result follows by Theorem \ref{thm:basic0}.
\end{proof}
\begin{proof} [Proof of Theorem \ref{theorem_pertubHyperbolic}]
Under the assumptions on the Euler characteristic by
~\cite[Theorem 1.10]{cite_SavelyevGromovFuller}
for any $\beta$-taut $g$ on $X$:
\begin{equation*}
\operatorname {GWF}(g, id, \beta, 1) = F (g, \beta ) = \chi
^{S ^{1}}(L _{\beta }X) \neq 0.
\end{equation*}
Then for $\phi $ as in the hypothesis, let $\{\phi _{t}\}
_{t \in [0,1]}$ be a homotopy between $id$ and $\phi$, with
each $\phi _{t}$ an isometry of $g$. Then $\{(g, \phi _{t})\}
_{t}$ clearly furnishes an $\mathcal{E} $-homotopy
between $(g, id)$ and $(g, \phi)$.
So that $\operatorname
{GWF}(g, \phi, \beta, 1) \neq 0$, by Theorem
\ref{thm:generalizationGWF}.
\end{proof}
\begin{proof} [Proof of Corollary \ref{corollary_hyperbolic}]
If $X$ admits a complete metric of negative curvature, then 
by the proof of \cite[Theorem
1.14]{cite_SavelyevGromovFuller} we may find a not a power
class $\beta' \in \pi _{1} ^{inc} (X)$, s.t. $\beta $ has
a representative which is a $k$-cover of a representative of
$\beta '$. Also,
\begin{align*}
	 \chi ^{S ^{1}}(L _{\beta '} X) & =  \chi(L _{\beta '}X/S
	^{1}) \quad \text{as the action is free by the condition that
	$\beta' $ is not a power} \\ 
	& = 1,
\end{align*}
where the last equality is immediate from the hypothesis
that $X$ admits a complete metric of negative curvature and
classical Morse theory, see proof of
~\cite[Theorem 1.10]{cite_SavelyevGromovFuller}.  Now, if
$g$ is any other complete metric on $X$ with non-positive
curvature then in particular it is $\beta '$-taut. Then by
Theorem \ref{theorem_pertubHyperbolic}, $\phi$ has a charge
one, class $\beta' $ fixed geodesic string. It readily
follows that $\phi $ also has a class $\beta $, charge one
fixed geodesic string. 
\end{proof}

\begin{proof} [Proof of Theorem \ref{thm:finitetypeRiemannian}]
Let $\{(g _{t}, \phi _{t})\}$ be an $\mathcal{E} $-homotopy
as in the hypothesis. And let $o \in \mathcal{O}  (g _{0}, \beta) $ be the 
unique and non-degenerate element. As $\phi
_{0,*} ^{n} (\beta) = \beta  $, it is immediate
that $o$ is a charge $n$ fixed geodesic string of $\phi _{0}$. 

Moreover, the moduli space $\overline{\mathcal{M}} ^n _{1,1}
(J ^{\lambda _{\phi _{0}}}, A ^{1}
_{\widetilde{\beta}}, \widetilde{\beta} )$ consists of one
point, and it is regular by the non-degeneracy of $o$. So
we conclude that $\operatorname {GWF}(g_0, \phi _{0}, \beta,
n) \neq 0$ ($\frac{1}{k}$ where $k$ is the multiplicity of
$o$.). Then by Theorem \ref{thm:generalizationGWF} we
get that $\operatorname {GWF}(g_1, \phi _{1}, \beta,
n) \neq 0$ and so we are done.
\end{proof}

\begin{proof} [Proof of Corollary \ref{cor:Riemannian}]
Let $g$ and $\beta \in \pi _{1} ^{inc} (X) $ be as in the hypothesis. By assumption the
corresponding unit cotangent bundle $(C, \lambda )$ is
definite type in class $\widetilde{\beta} $. If $\phi $ is
some isometry of $g$ homotopic through isometries to the
identity, then the induced strict contactomorphism
$\widetilde{\phi} $ is
homotopic through strict contactomorphisms to the $id$, and
so by Theorem \ref{thm:infinitetype} $\widetilde{\phi
} $ has a class $\widetilde{\beta } $ fixed Reeb
string. So that $\phi $ has a class $\beta $ fixed
geodesic string.
\end{proof}

\begin{proof} [Proof of Theorem \ref{thm:basic1}] 
Let $g _{0}$ be a complete metric on $X$ with
a unique and non-degenerate class $\beta \in \pi _{1}
^{inc}(X)$ geodesic string.
Let $\lambda _{0} = \lambda _{g _{0}}$ be the Liouville
1-form on the $g _{0}$-unit contangent bundle $C$ of $X$.
Let $g$ be as in the hypothesis and let $\lambda  _{1}$ be Liouville 1-form on $C$ corresponding to $g$. 

Let $\{\lambda _{t}\}$, $t \in [0,1] $,  be a smooth homotopy between $\lambda _{0}, \lambda _{1}$.  
We may in addition assume that this homotopy is constant
near the end points. We then get a family $\{(\lambda
_{t}, \alpha)\}$ of first kind integral lcs structures on $C \times S ^{1}$. 

Now let $\{\widetilde{\phi} _{t}\}$, $t \in [0,1] $ be a smooth homotopy
of strict contactomorphisms of $(C, \lambda _{1}) $ with
$\widetilde{\phi}
_{0} = id$, corresponding to a homotopy $\{\phi _{t}\}$ of
isometries of $X,g$, with $\phi _{0} = id$.
We may suppose that
$\{\widetilde{\phi} _{t}\}$ is constant near end points. As in the Proof
of Theorem \ref{thm:basic0}, this gives a smooth
fibration over $\widetilde{M} \to [0,1] $. And as before we
get a family $\{(\lambda'' _{t}, \alpha) \} $, $t \in [0,1]
$, of first kind integral lcs structures on $C \times
S ^{1}$, s.t.: 
\begin{itemize}
\item $(\lambda'' _{0}, \alpha)$ is the lcs-fication of
$\lambda _{0}$. 
\item For each $t$ $(\lambda'' _{t}, \alpha)  $ is isomorphic to
the mapping torus structure $(\lambda _{\widetilde{\phi} _{t}}, \alpha
) $. 
\end{itemize}

Let	 $\{(\lambda' _{t}, \alpha )\}$ be the concatenation of the families
$\{(\lambda _{t}, \alpha ) \}$, $\{(\lambda'' _{t}, \alpha
) \}$, i.e.:
\begin{equation*}
\lambda' _{t} = \begin{cases}
	\lambda _{2t}, &\text{ if } t \in [0, \frac{1}{2}]\\
	\lambda'' _{2t-1}, &\text{ if } t \in [\frac{1}{2}, 1].\\
\end{cases}
\end{equation*}

Let $\{J'  _{t}\}$, $t \in [0,1] $  be a family 
of almost complex structures on $M=C \times S ^{1}$ s.t. $J'
_{t}$ is $(\lambda' _{t}, \alpha )$-admissible for each $t$.
Let $\widetilde{\beta } $ be the lift of $ \beta \in
\pi _{1}  ^{inc} (X)$, as in Section \ref{sec:Definition of GWF}.
Now, by construction and by Theorem \ref{thm:GWFullerMain}
$$GW ^{1} _{1,1} (J' _{0}, A _{\widetilde{\beta}}, \beta )
= i (R ^{\lambda _{g _{0}}}, \widetilde{\beta}) \neq 0, $$ 
with the last inequality due to the assumption that there is
a unique non-degenerate $g _{0}$-geodesic string in class $\beta $.
Then by Lemma \ref{thm:welldefined}, we get that one
of the following holds:
\begin{itemize}
	\item $(\lambda _{\widetilde{\phi}}, \alpha ) $ has an elliptic
charge 1, class $A _{\widetilde{\beta }}$ Reeb 2-curve $u$,
and hence by part one of the Proposition
\ref{prop:Topembedding}, $\widetilde{\phi} $ has
a charge 1 fixed Reeb string in class $\widetilde{\beta } $,
and so $\phi $ has a charge 1 fixed geodesic in class $\beta
$.
\item The family $\{J ' _{t}\}$,
$t \in [0,1] $ has an essential right holomorphic sky catastrophy,
of charge 1, class $ A _{\widetilde{\beta}}$ curves. 
\end{itemize}

Suppose that the latter holds. By the admissibility
condition and by Lemma \ref{lemma:Reeb}, if
$$(u, t) \in \overline{\mathcal{M}} ^{1} _{1,1}
(\{J'  _{t} \}, A _{\widetilde{\beta}}, \beta )$$ then $u$ is a charge 1 elliptic Reeb 2-curve. 
Let $\{X _{t}\}$, $t \in [0,1] $, be the smooth
family of vector fields satisfying $X _{t} = R ^{\lambda
_{2t}}$ for $t \in [0, \frac{1}{2}] $,  $X _{t}
= R ^{\lambda _{1}}$ for $t \in [\frac{1}{2}, 1] $.
Analogously to part \ref{part4_propTopEmbedding} of Proposition
\ref{prop:Topembedding}, we have a proper topological
embedding $$emb: \overline{\mathcal{M}} ^{1} _{1,1}
(\{J'  _{t} \}, A _{\widetilde{\beta}}, \beta) \to \mathcal{O} (\{X _{t}\} , \widetilde{\beta}),$$
 
It follows, by the previous hypothesis, that the family  $\{X
_{t}\} _{t \in [0,1] }$  has a sky catastrophe in class $\beta=0$. 
In addition, this sky catastrophe must be
essential,  as otherwise the original holomorphic sky
catastrophe would not be essential. 

\end{proof}

\begin{proof} [Proof of Theorem \ref{thm:generalizationGWF}]
We have to prove invariance of the counts.
Let $\{(g _{t}, \phi _{t})\} _{t \in [0,1]}$,  be an
$\mathcal{E} $-homotopy. 
Let $(C, \lambda _{t} = \lambda _{g _{t}}) $, be the $g
_{t}$-unit cotangent bundle of $X$, and $\lambda _{g _{t}}$
the Liouville 1-form. (Fixing an implicit identification of
the unit cotangent bundles with a fixed manifold $C$.)
Let $\widetilde{\beta}  \in \pi _{1} (C) $, be the lift of
a class $ \beta \in \pi _{1} ^{inc} (X)$ as in Section
\ref{sec:Definition of GWF}. Denote by $(M _{t}, \lambda
_{t}, \alpha _{t})$ the mapping torus of $\widetilde{\phi
} _{t}$ action on $(C, \lambda _{t})$, where
$\widetilde{\phi} _{t} $ is the strict contactomorphism
induced by the isometry $\phi _{t}$. Finally, let $\{J
_{t}\} _{t \in [0,1]}$,  be a smooth family with $J _{t}$  $(\lambda _{t}, \alpha _{t})$-admissible almost complex structures on $M _{t}$.

By the tautness assumption the family $\{R ^{\lambda
_{t}}\}$ has no sky catastrophe in the class
$\widetilde{\beta } $.
Then by the third part of Proposition
\ref{prop:Topembedding}, $\{J _{t}\}$ has no sky catastrophe in class in $A
_{\widetilde{\beta} }$, and so by Theorem
\ref{thm:welldefined} 
we have $$\operatorname {GWF}  (g
_{0}, \phi _{0}, \beta, n) = \operatorname {GWF}  (g
_{1}, \phi _{1}, \beta, n) $$ and
we are done.
\end{proof}


\begin{proof} [Proof of Theorem 
\ref{thm:Reeb1curves}] 
   Suppose that $u: \Sigma \to M $ is 
   an immersed Reeb 2-curve, we then show that $M$ 
   also has a Reeb 1-curve. Let $\widetilde{u}: 
   \widetilde{\Sigma} \to M  $ be the normalization 
   of $u$, so that $\widetilde{u} $ is an 
   immersion. We have a pair of transverse 1-distributions $D 
   _{1}= \widetilde{u}  ^{*}\mathbb{R} ^{} \langle X _{\alpha }  \rangle   
   $, $D _{2}=\widetilde{u} ^{*} \mathbb{R} ^{} \langle X 
   _{\lambda }  \rangle $ on $\widetilde{\Sigma } $. We 
   may then find an embedded path $\gamma: 
   [0,1] \to \widetilde{\Sigma }   $, tangent to 
   $D _{1}$  s.t. 
   $\lambda (\gamma' (t)) >0 $, $\forall t \in 
   [0,1] $, and s.t.   
   $\gamma (0) $ and $\gamma (1) $ are on a 
   leaf of $D _{2}$.  It is then simple to obtain from 
   this a Reeb 1-curve $o$, by joining the end points of 
   $\gamma $ by an embedded path tangent to $D 
   _{2}$, and perturbing, see Figure \ref{figure:smooth}.
   \begin{figure}[h]
 \includegraphics[width=1.4in]{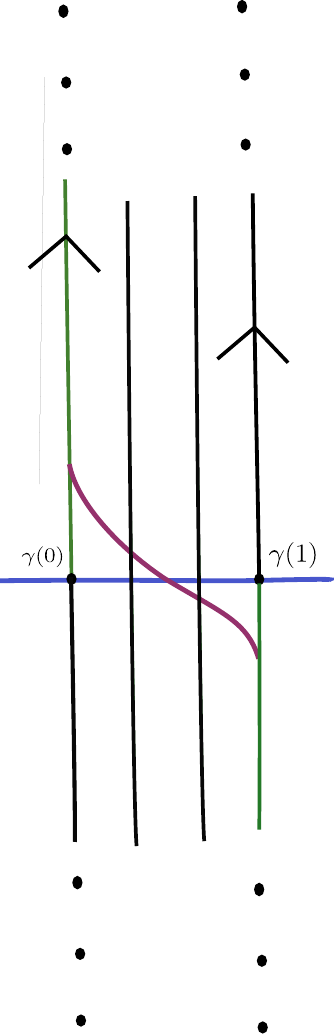}
 \caption {The green shaded path is $\gamma $, the 
 indicated orientation is given by $u ^{*}\lambda$, the 
 $D _{1}$ folliation is shaded in black, the 
 $D _{2}$ folliation is shaded in blue. The purple 
 segment is part of the loop 
 ${o}: S ^{1} \to \Sigma  $, which is 
 is smooth and satisfies $\lambda 
 ({o}' (t)  )>0 $ for all $t$.
 } \label{figure:smooth}
\end{figure}   
This proves the first part of the theorem. 

To prove the second part, suppose that $u: \Sigma \to M$ is an immersed 
elliptic Reeb 2-curve. Suppose that $u$ is not 
normal. Let $\widetilde{u}: \widetilde{\Sigma } \to M$ be its normalization.  Then $\widetilde{\Sigma} $ has a genus 0 
component $\mathcal{S}$. So that $\widetilde{u}: 
\mathcal{S} \simeq \mathbb{CP} ^{1}  \to M$ is 
immersed. The distribution $D 
_{1} = \widetilde{u}  ^{*}\mathbb{R} ^{} \langle X 
_{\alpha}  \rangle$, as appearing above, is 
then a $\widetilde{u} ^{*}\lambda$-oriented 1-dimensional distribution on $ \mathbb{CP} ^{1}  $ which is impossible.
\end{proof}

\begin{appendices} 
\section{Fuller index} \label{appendix:Fuller} Let $X$ be
a complete vector field without zeros on a smooth manifold $M$. 
Set 
\begin{equation} S (X, \beta) = 
   \{o \in L _{\beta} M \,\left .  \right | \exists p \in 
   (0, \infty), \, \text{ $o: \mathbb{R}/\mathbb{Z} \to M $ is a
   periodic orbit of $p X $} \},
\end{equation}
where $L _{\beta} M  $ denotes
the free homotopy class $\beta$ component of the 
free loop space $$LM = \{o: S ^{1} \to M \,|\,  \text{$o$ is
smooth} \} .$$ And where recall that $S ^{1}
= \mathbb{R}/\mathbb{Z}  $.   The above $p$ is uniquely
determined and we denote it by $p (o) $ called the period of
$o$.

There is a natural $S ^{1}$ reparametrization action on $S
(X, \beta )$:  $t \cdot o$ is the loop $t \cdot o(\tau) = o (t + \tau)  $.
The elements of $\mathcal{O} (X, \beta ) := S (X, \beta )/S ^{1} $ will be 
called \emph{orbit strings}.
Slightly abusing notation we just write $o$  for the
equivalence class of $o$. 

The multiplicity $m (o)$ of an orbit string is
the ratio $p (o) /l$ for $l>0$ the period of a simple orbit 
covered by $o$. 

We want a kind of fixed point index of an open compact
subset $N \subset \mathcal{O} (X, \beta )$, which counts orbit
strings $o$ with certain weights.
Assume for simplicity that  $N \subset \mathcal{O} (X) $ is 
finite. (Otherwise, for a general open compact $N \subset
\mathcal{O} (X, \beta )$, we need to perturb.) Then
to such an $(N,X, \beta)$
Fuller associates an index: 
\begin{equation*}
   i (N,X, \beta) = \sum _{o \in N}\frac{1}{m
   (o)} i (o),
\end{equation*}
 where $i (o)$ is the fixed point index of the time $p (o) $ return map of the flow of $X$ with respect to
a local surface of section in $M$ transverse to the image of $o$. 

Fuller then shows that $i (N, X, \beta )$ has the following invariance property.
For a continuous homotopy $\{X _{t}
\}$,  $t \in [0,1]$ set \begin{equation*}  S ( \{X _{t} \}, \beta) = 
   \{(o, t) \in L _{\beta} M \times [0,1] \,|\,  
	  \text{$o \in S (X _{t})$}\}.
\end{equation*}
And given a continuous homotopy $\{X _{t}
\}$, $X _{0} =X $, $t \in [0,1]$, suppose that 
$\widetilde{N} $ is an open compact subset of $S 
(\{X _{t} \}, \beta ) / S ^{1}$, such that $$\widetilde{N} 
\cap \left (L _{\beta }M \times \{0\} 
\right) / S ^{1} =N.
$$ Then if $$N_1 = \widetilde{N} \cap \left (L _{\beta }M 
  \times \{1\} \right) / S 
^{1}$$ we have 
\begin{equation*}
i (N, X, \beta ) = i (N_1, X_1, \beta).
\end{equation*}

In the case where $X$ is the $R ^{\lambda} $-Reeb vector 
field on a contact manifold $(C ^{2n+1} , \lambda )$,
and if $o$ is
non-degenerate, we have: 
\begin{equation} \label{eq:conleyzenhnder}
i (o) = \sign \Det (\Id|
   _{\xi (x)}  - F _{p (o), *}
^{\lambda}| _{\xi (x)}   ) = (-1)^{CZ (o)-n},
\end{equation}
where $F _{p (o), *} ^{\lambda}$ is the differential 
at $x$ of the time $p (o) $ flow map of $R ^{\lambda} 
$, and where $CZ (o)$ is the Conley-Zehnder 
index, see \cite{cite_RobbinSalamonTheMaslovindexforpaths.}.

There is also an extended Fuller index $i (X, 
\beta) \in \mathbb{Q} \sqcup \{\pm \infty\}$, for 
certain $X$ having definite type. This is 
constructed in \cite{cite_SavelyevFuller}, and is 
conceptually analogous to the extended 
Gromov-Witten invariant described in this paper. 

The following is a version of the definition of sky catastrophes first
appearing in Savelyev~\cite{cite_SavelyevFuller},
generalizing a notion commonly called a ``blue sky
catastrophe'', see
Shilnikov-Turaev~\cite{cite_ShilnikovTuraevBlueSky}.  
\begin{definition}\label{def:bluesky}
Let  $\{X _{t} \}$, $t \in [0,1] $ be a continuous family of
non-zero, complete smooth vector fields on a closed manifold
$M$, and let $\beta \in \pi _{1} ^{inc} (X)$.  And let $S
(\{X _{t}\}) $ be as above.
We say that $\{X _{t}\}$  has a \textbf{\emph{right sky catastrophe in class $\beta$}}, if
there is an element $$y \in S  (X _{0}, \beta)
\subset S (\{X
_{t}\},  \beta)  $$ so that there is no open compact subset
of $S
(\{X _{t}\}, \beta) $ containing $y$. 
We say that $\{X _{t}\}$  has a \textbf{\emph{left sky catastrophe in class $\beta$}}, if
there is an element $$y \in S (X _{1}, \beta)
\subset S (\{X
_{t}\},  \beta)  $$ so that there is no open compact subset
of $S (\{X _{t}\}, \beta) $ containing $y$. 
We say that $\{X _{t}\}$  has a \textbf{\emph{sky
catastrophe in class $\beta$}}, if it has either left or
right sky catastrophe in class $\beta $.
\end{definition}
\begin{definition}\label{def:essentialReebsky}
In the case that $X _{t}  = R ^{\lambda _{t}}$ for $\{\lambda _{t}\}$, $t \in [0,1] $ smoothly varying, we say that a sky
catastrophe of Reeb vector fields $\{X _{t}\}$ is
\textbf{\emph{essential}} if the condition of the definition
above holds for any family $\{X' _{t} = R ^{\lambda' _{t}}\}$ satisfying $X' _{0} = X _{0}$ and $X' _{1} = X _{1}$, and such that $\{\lambda' _{t}\}$ is smooth.
\end{definition}


\section{Remark on multiplicity} \label{sec:GromovWittenprelims}
This is a small note on how one deals with curves having non-trivial isotropy groups, in the virtual fundamental class technology. We primarily need this for the proof of Theorem \ref{thm:GWFullerMain}.

Given a closed oriented orbifold $X$, with an orbibundle $E$ over $X$
Fukaya-Ono \cite{cite_FukayaOnoArnoldandGW} show how to construct
using multi-sections its rational homology Euler
class, which when $X$ represents the moduli space of some stable
curves, is the virtual moduli cycle $[X] ^{vir} $.  When this is in degree 0, the corresponding Gromov-Witten invariant is $\int _{[X] ^{vir} } 1. $
However, they  assume that their orbifolds are
effective. This assumption is not really necessary for the purpose of
construction of the Euler class but  is convenient for other technical reasons. A
different approach to the virtual fundamental class which emphasizes branched manifolds is used by
McDuff-Wehrheim, see for example McDuff
\cite{cite_McDuffNotesOnKuranishi},
\cite{cite_McDuffConstructingVirtualFundamentalClass} which
does not have the effectivity assumption, a similar use of
branched manifolds appears in
\cite{cite_CieliebakRieraSalamonEquivariantmoduli}. In the
case of a non-effective orbibundle $E \to X$  McDuff
\cite{cite_McDuffGroupidsMultisections}, constructs
a homological Euler class $e (E)$ using multi-sections,
which  extends the construction \cite{cite_FukayaOnoArnoldandGW}.  McDuff shows that this class $e (E)$ is
Poincare dual to the completely formally natural cohomological Euler class of
$E$, constructed by other authors. In other words there is a natural notion of a
homological Euler class of a possibly non-effective orbibundle.
We shall assume the following black box property of the virtual fundamental
class technology.
\begin{axiom} \label{axiom:GW} Suppose that the moduli space of stable maps is cleanly cut out, which means that it is represented by a (non-effective) orbifold $X$ with an orbifold obstruction bundle $E$, that is the bundle over $X$ of cokernel spaces of the linearized CR operators. Then the virtual fundamental class $[X]^ {vir} $ coincides with $e (E)$.
\end{axiom}
Given this axiom it does not matter to us which virtual moduli cycle technique we use. It is satisfied automatically by the construction of McDuff-Wehrheim,
(at the moment in genus 0, but surely extending).
It can be shown to be satisfied in the approach of John
Pardon~\cite{cite_PardonAlgebraicApproach}.
And it is satisfied by the construction of Fukaya-Oh-Ono-Ohta
\cite{cite_FOOOTechnicaldetails}, the latter is 
communicated to me by Kaoru Ono. When $X$ is 0-dimensional this does follow
 immediately by the construction in
 \cite{cite_FukayaOnoArnoldandGW}, taking any effective Kuranishi neighborhood at the isolated points of $X$, (this actually suffices for our paper.)

As a special case most relevant to us here, suppose we have a moduli space
of elliptic curves in $X$, which is regular with
expected dimension 0. Then its underlying space is a collection of oriented points.
However, as some curves are multiply covered, and so have isotropy groups,  we must treat this is a non-effective 0 dimensional oriented orbifold.
The contribution of each curve $[u]$ to the Gromov-Witten invariant $\int
_{[X] ^{vir} } 1 $ is $\frac{\pm 1}{[\Gamma ([u])]}$, where $[\Gamma ([u])]$ is the order of the isotropy group $\Gamma ([u])$ of $[u]$, 
in the McDuff-Wehrheim setup this is explained in \cite[Section
5]{cite_McDuffNotesOnKuranishi}. In the setup of Fukaya-Ono
\cite{cite_FukayaOnoArnoldandGW} we may readily calculate to get the same thing taking any effective Kuranishi neighborhood at the isolated points of $X$. 
\end{appendices}
\section{Acknowledgements} 
Yong-Geun Oh for discussions on related topics, and for an invitation to IBS-CGP, Korea. Egor Shelukhin and Helmut Hofer for an invitation to the IAS
where I had a chance to present and discuss early versions
of the work.
I also thank Kevin Sackel, Kaoru Ono, Baptiste
Chantraine, Emmy Murphy, Viktor Ginzburg, Yael Karshon,
John Pardon, Spencer Cattalani and Dusa McDuff for related discussions.
\bibliographystyle{siam} 
\bibliography{C:/Users/yasha/texmf/bibtex/bib/link.bib} 

\end{document}